\documentclass[11pt,amsfonts]{article}
\usepackage{amsmath,amsfonts,amsthm,amssymb,color}
\usepackage[T1]{fontenc}
\usepackage{enumerate}
\usepackage{graphicx}
\usepackage{float}
\usepackage[ruled,vlined]{algorithm2e}
\usepackage{amssymb}
\usepackage{amsmath}
\usepackage{color}
\usepackage{layout}
\usepackage{amsfonts}
\usepackage{dsfont}
\usepackage{cite}
\usepackage{color}
\usepackage{listings}
\usepackage{anysize}
\usepackage[noend]{algpseudocode}
\usepackage{ulem}
\usepackage{latexsym}
\usepackage{ragged2e}
\usepackage{pdfsync}
\usepackage{hyperref}

\newtheorem{theorem}{Theorem}[section]
\newtheorem{prop}[theorem]{Proposition}
\newtheorem{lemma}[theorem]{Lemma}

\newtheorem{remark}[theorem]{Remark}
\newtheorem{hypothesis}{Hypothesis}

\def\real{{\mathord{{\rm I\kern-2.8pt R}}}}        % Fake blackboard bold R.
\def\inte{{\mathord{{\rm I\kern-2.8pt N}}}}

\def\sZZ{{\rm Z\kern-2.8ptem{}Z}}

\def\z{{\mathchoice
		{\sZZ}
		{\sZZ}
		{\rm Z\kern-0.30em{}Z}
		{\rm Z\kern-0.25em{}Z} }}
\def\sQQ{{\kern 0.27em \vrule height1.45ex width0.03em depth0em
		\kern-0.30em \rm Q}}
\def\qu{{\mathchoice
		{\sQQ}
		{\sQQ}
		{\kern 0.225em \vrule height1.05ex width0.025em depth0em \kern-0.25em \rm Q}
		{\kern 0.180em \vrule height0.78ex width0.020em depth0em \kern-0.20em \rm Q}
}}
\def\sCC{{\kern 0.27em \vrule height1.45ex width0.03em depth0em
		\kern-0.30em \rm C}}
\def\complex{{\mathchoice
		{\sCC}
		{\sCC}
		{\kern 0.225em \vrule height1.05ex width0.025em depth0em \kern-0.25em \rm C}
		{\kern 0.180em \vrule height0.78ex width0.020em depth0em \kern-0.20em \rm C}
}}

\font\tenmath=msbm10 \font\sevenmath=msbm7 \font\fivemath=msbm5
\newfam\mathfam \textfont\mathfam=\tenmath
\scriptfont\mathfam=\sevenmath \scriptscriptfont\mathfam=\fivemath

\newcommand{\ignore}[1]{}
\numberwithin{equation}{section}

\begin{document}
	
	\renewcommand{\thefootnote}{\fnsymbol{footnote}}

	\title{Euler scheme for stochastic  functional differential equations driven by fractional Brownian motion}
	
	\author{ Johanna Garz\'on$^{1,a}$   \ \ \ \ Jorge A. Le\'on$^{2,b}$  \ \ \ \ Jorge Lozada$^{2,c}$ \ \ \ \ Soledad Torres$^{3,d}$ \\
		%-----------------------------------------------------------
		\small $^{1}$ Departamento de Matem\'aticas, 
		Facultad de Ciencias,\\
		\small Universidad Nacional de Colombia,\\
		\small $^a$mjgarzonm@unal.edu.co\\ 
		%------------------------------------------------------- 
		\small $^{2}$ Departamento de Control Autom\'atico, Cinvestav,\\ 
		\small			Ciudad de M\'exico, Mexico,\\
		\small  $^b${joleon@cinvestav.mx}, \quad $^c${jorge.lozada@cinvestav.mx} \\ 
		%-------------------------------------------------------  
		\small $^{3}$ CIMFAV, Facultad de Ingeniería, Universidad de Valparaíso,\\ 
		\small Valparaíso, Chile,\\
		\small $^d$soledad.torres@uv.cl\\ 
	}

	\date{}
	
	\maketitle
	\begin{abstract}
		In this paper, we apply  rough paths techniques  to 
		provide an approximation of the solution of stochastic functional differential equations driven by fractional Brownian motion with Hurst parameter $H>1/2$. Here, the  involved stochastic integral is the Young one
		and the coefficient is evaluated in the set of $\lambda$-H\"older continuous functions on $[-\tau,0]$, for some suitable
		$\tau>0$ and $\lambda\in(1/2,H)$. The rate of convergence of our scheme is $1/n^{\gamma}$, for any $\gamma<2\lambda-1$. Also, numerical simulations are provided to illustrate our theoretical results.
	\end{abstract}
	\vskip0.3cm
	
	{\bf 2010 AMS Classification Numbers:} 34K07, 65C30.
	\vskip0.3cm
	{\bf Key Words and Phrases}:  Euler-Maruyama method, fractional Brownian motion, Fr\'echet derivative, H\"older continuous functions, 
	stochastic functional differential equations, Young integral.

	\section{Introduction}
	Stochastic functional differential equations (SFDEs) can be defined as a class of stochastic differential equations for which the coefficients depend on the history of the process itself. That is, these coefficients are defined on a function space, giving rise to the name of this class of equations.
	
	In this article, we consider a SFDE of the form
	
	\begin{eqnarray}
		\label{eqi1}
		dX(t)&=& f(\overline{X}_t)dB^H(t), \quad t\in [0, T], \\
		\overline{X}_0&=& \xi. \notag
	\end{eqnarray}
	Here, $B^H=\{B^H(t):t\in[0,T]\}$ is a fractional Brownian motion with Hurst parameter $H>1/2$ and the involved stochastic integral is the Young one introduced by Young\cite{Young} (see also Z\"ahle \cite{zahle}). The initial condition $\overline{X}_0=\xi$ is  a $C^\gamma([-\tau, 0]; \mathbb{R})$-valued random variable,
	where $C^\gamma([-\tau, 0]; \mathbb{R})$
	is the set of $\gamma$-H\"older continuous functions from $[-\tau, 0]$ to $\mathbb{R}$, and $f: C^\gamma([-\tau, 0]; \mathbb{R}) \to \mathbb{R}$.  For each $t\in [0, T]$, $\overline{X}_t\in {\cal{C}}^\gamma ([-\tau,0] ; \mathbb{R})$ denotes \textit{the segment} of $X$ in the past time interval $[t-\tau, t]$, which represents the functional memory or history of the solution $X$ at time $t$. It is defined by
	\begin{equation*}
		\label{eqdefXsub}
		\overline{X}_t(\theta)=X({t+\theta}), \quad  -\tau\le \theta \leq 0.
	\end{equation*}
	In particular, \textit{the segment} $\overline{X}_0=\xi$ contains the information of the process before $t = 0$.
	
	For Brownian motion (i.e. $H=1/2$), SFDEs were introduced by Itô and Nisio \cite{IN} in 1964. This type of equation has been extensively studied (see Bao et al.\cite{Bao} and Mohammed\cite{Moh, Moh2}, and  references therein), particularly the case of stochastic delay differential equations (SDDEs), which are a special case of SFDEs where \textit{the segment} $\overline{X}_t$ is the constant function given by $\overline{X}_t(\theta)=X(t-\tau)$ (according to the notation given above). FSDEs and SDDEs are useful for modeling random phenomena that depend not only on the present but also on the past states of the system and appear in various applications in fields such as biology, finance, physics, medicine, and population dynamics \cite{Anh, arriojas, Bahar, Brett, Mao3, Miekisz, smith}.	
	
	Since obtaining explicit solutions of SFDEs is not always possible, it is important to have numerical methods to approximate their solutions. When noise is still given by a Brownian motion, K\"ochler and Platen \cite{kuchler} give an approach to the derivation of discrete-time approximations for solutions of SDDE where the convergence is in a strong sense.  Mao and Sabanis \cite{Mao5} show that the Euler-Maruyama numerical solutions converge to the true solutions of SDDE under the local Lipschitz condition. Buckwar \cite{Buckwar} obtained approximations of strong solutions of an SDDE. Guo et al. \cite{Guo} study the strong convergence of the explicit numerical method for highly nonlinear SDDE using the truncated Euler-Maruyama method. Mao \cite{mao} established the strong mean-square convergence theory for numerical solutions of SFDE under the local Lipschitz condition and the linear growth condition. Mao \cite{Mao4} developed convergence in probability of numerical solutions of SDDE under Khasminskii-type conditions.  Milosevic \cite{Milosevic} studied the convergence of the approximate Euler-Maruyama solutions to the nonlinear class of SDDE. Zhou et al. \cite{Zhou} established the convergence of numerical solutions to a neutral stochastic delay differential equation with Markovian switching.

	On the other hand, our interest in working with noise driven by fractional Brownian motion is motivated by its properties, such as self-similarity and long-range dependence, as well as their applications to problems in several fields such as econometrics, internet traffic, geophysics, statistical physics, neuroscience, DNA sequences, telecommunications, finance, electrical engineering, and hydrology (see, for instance,  Biagini \cite{biaginiI}, Denk et al. \cite{DMS}, Kuo and Xie\cite{KS},
	Nualart \cite{nualart} and {\O}ksendal \cite{oksendalI}).\\
	
	Fractional Brownian motion (fBm) $B^H$, with parameter $H\in (0,1)$, is a centered Gaussian process with covariance function given by
	\begin{align*}
		%\label{eqcovarianza}
		\mathbb{E}(B^H(t)B^H(s))= \frac{1}{2} \left[ s^{2H}+t^{2H}- |t-s|^{2H} \right],\quad s,t \in[0,T].
	\end{align*}
	The Hurst parameter $H$  divides the fractional Brownian motion into three regimes corresponding to $H = 1/2$  with independent increments (i.e., Brownian motion  case), $0 < H < 1/2$, which means that the process is persistent, and when $1/2 < H < 1$, the process is antipersistent. In this article, we will focus only on the case $H>1/2$. See Le\'on \cite{Leon}, Mishura\cite{mishuraI} or Nualart \cite{nualart} for a general background.
	
	The classical Itô theory cannot be used to construct stochastic calculus with respect to fBm, since it is not a semi-martingale for $H\neq 1/2$. Consequently, various theories in the field of stochastic integration with respect to fractional Brownian motion have been developed \cite{ biaginiI, Leon, mishuraI, nualart,  ruso, zahle}. In this article, we will work with the pathwise Young integral, which is a generalized Stieltjes integral \cite{Young}.
	
	Existence and uniqueness of solutions to fractional stochastic differential equations (i.e,  stochastic differential equations driven by fBm) have been studied (Coutin and Qian\cite{coutin}, Hu and Nualart 
	\cite{hu}, Nualart and R\u{a}\c{s}canu \cite{nualart2}, {among others}),  and methods for approximating solutions have been devised (e.g.  Hu et al. \cite{Hu}, Liu and Tindel \cite{LiuT}, Mishura \cite{mishuraI}, Neuenkirch and Nourdin \cite{andreas}). For stochastic functional differential equations driven by fractional Brownian motion, there exist only a few papers published in this field. For the delay case, Ferrante and Rovira, in \cite{ferrante1}, established an existence and uniqueness result for SDDE driven by fBm with Hurst parameter $H>1/2$, when the coefficients are sufficiently regular,  and in \cite{ferrante2}, they studied the convergence of solution of SDDE when the delay goes to zero. Also, Neuenkirch and Nourdin \cite{andreas} studied existence and uniqueness for SDDE but for fBm with $H>1/3$ using rough paths theory. Boufoussi and Hajji \cite{Boufoussi} give a result of global existence and uniqueness for the solution of SFDE  driven by a fBm with $H > 1/2$ and non-constant delay, also they study the dependence of the solution on the initial condition. In Boufoussi et al. \cite{Boufoussi2}, authors deal with SFDE  in a Hilbert space. Lakhel and Mckibeen \cite{Lakhel} studied existence of mild solutions for neutral SFDE, in Hilbert space, by using the Wiener integral. Wilathgamuwa \cite{Wilathgamuwa} proves existence and uniqueness for SFDE using the pathwise Riemann-Stieltjes integral.  
	
	It is worth mentioning that, in this article, we focus only on the one-dimensional case since the extension to $\mathbb{R}^d$ ($d>1$) is natural from the one-dimensional equation by introducing appropriate indices. Specifically, equation \eqref{eqi1} generalizes to
	$$dX(t)= \sum_{i=1}^d f^i(\overline{X}_t)d(B^H)^i(t),$$
	where $f:\mathbb{R}^m\to \mathbb{R}^{m\times d}$, $f^i$ denotes the i-th column of the matrix $f(x)$  and $(B^H)^i$ is the i-th component of  $d-$dimensional fractional Brownian motion. Hence, the transition from the one-dimensional to the multidimensional analysis is straightforward. Therefore, we restrict our study to the one-dimensional case without loss of generality.

	Moreover, we consider equation \eqref{eqi1} without a drift term. The absence of drift simplifies the presentation and allows us to focus on the main analytical challenges posed by the delay and the irregularity of the noise. This simplification does not entail a loss of generality since the drift component is not essential to the structure of the problem. In fact, the drift can be naturally absorbed into the noise by augmenting the driving signal with an additional deterministic component. This allows the drift to be treated as part of the stochastic input, thereby avoiding additional complications in both the theoretical and numerical analysis. With this reformulation, we are left with the discretization of the following equation over the interval $[0,T]$.
	
	To the best of our knowledge, the only work on the numerical approximation of SFDEs driven by a fractional Brownian motion is that of by Garz\'on et al. \cite{Garzon}, where the authors propose a discrete-time approximation for the solution of SDDEs and determine its rate of convergence. Following the ideas presented in \cite{Garzon} and \cite{mao} (in the Brownian motion case), in this article we define an Euler-Maruyama scheme $X^{n}$ to approximate the unique solution of equation \eqref{eqi1} and establish its rate of convergence. 
	In this article, we extend the results of Garz\'on et al. \cite{Garzon} by considering a more general delay structure. Instead of focusing on a single fixed delay, we incorporate the full discrete mesh of delay points, reflecting the fact that the delay arises from a functional dependence on past values. We define the Euler scheme over the entire continuous interval and, employing rough path techniques, establish the convergence of the scheme to a functional stochastic differential equation driven by fractional noise. In this case, the convergence rate found is of order $1/n^{\gamma}$, for any $\gamma<2\lambda-1$, analogous to the result obtained by \cite{Garzon}.
	
	This paper is organized in the following way. Some preliminaries are presented in Section \ref{sec:preliminaries}. Euler-Maruyama scheme is defined in Section \ref{Euler}. The rate of convergence of the numerical scheme is established in Section \ref{sec:conv}. A numerical example is given in Section \ref{sec:examples}.

	\section{Preliminaries}\label{sec:preliminaries}
	In this section, we introduce the framework and the notation that we use in this paper.
	
	Throughout the paper, for $T>0$ fixed, we suppose that we have a fractional Brownian motion $B^H=\{B^H(t):
	t\in[0,T]\}$ with Hurst parameter $H\in(1/2,1)$ defined on a complete probability space $(\Omega, \mathcal{F}, P)$. The fractional Brownian motion $B^H$ is a zero-mean Gaussian process with covariance function
	$$	\mathbb{E}(B^H(t)B^H(s))= \frac{1}{2} \left[ s^{2H}+t^{2H}- |t-s|^{2H} \right],\quad s,t\in[0,T].$$
	\subsection{H\"older continuous function and Young integral}
	The purpose of this section is to introduce the Young integral for H\"older continuous functions and some properties of last family of functions that we need in the remainder of this article. For a detailed exposition on Young integration, the reader can consult Le\'on \cite{Leon}, Le\'on and Tindel \cite{LeonTindel}, or Z\"ahle \cite{zahle}.
	
	Here, we fix a partition  $\pi:\{0=t_{0}<t_{1}<\cdots<t_{n}=T\}$  of the interval $[0,T]$. For $s,t\in [0,T]$, $s<t$, we denote by $[\![ s, t ]\!]$ the discrete interval $\pi \cap [s, t]$. Moreover, for $k\in\mathbb{N}$ and $I$ either the interval $[s,t]$ or the discrete interval $[\![ s, t ]\!]$, ${\mathcal S}_{k}(I)$ represents the simplex $\{(s_1, \cdots, s_{k})\in I^k: s_1\le \cdots \le s_k\}$.

	Now, we introduce some basic concepts on Young integration theory.
	
	For a vector space $V$ and two functions $f\in C([0,T],V)$ and $g\in C({\mathcal S}_{2}([0,T]),V)$, we set 
	\begin{equation*}\label{eq:def-delta}
		\delta f(s,t)= f({t})-f({s})
		\ \text{and}\
		\delta g(s,u,t) = g(s,t)-g({s,u})-g({u,t}),\quad (s,u,t)\in{\mathcal S}_{3}([0,T]).
	\end{equation*}

	We denote by $C^\gamma([a, b]; \mathbb{R})$ the set of $\gamma$-H\"older continuous functions from $[a, b]$ to $\mathbb{R}$ with $\gamma \in (0,1)$. In this space, consider the seminorm 
	$$\left\|\phi\right\|_{\gamma, [a, b]}: =  \sup_{x,y\in [a, b]  \atop x\neq y} \frac{|\phi(x)-\phi(y)|}{|x-y|^\gamma}$$
	and the supremum norm
	$$\left\|\phi\right\|_{[a, b]}=  \sup_{x\in [a, b]}{|\phi(x)|}.$$
	For the sake of the notation, we will omit the involved  interval in these norms if it is obvious.
	
	Now, we are ready to introduce the integral considered in equation \eqref{eqi1}, which is introduced
	by Young \cite{Young}.
	\begin{prop}[Young integral for H\"older continuous functions]
		\label{prop:young-integral}
		Let $\alpha, \  \beta \in (0, 1]$ such that $\alpha+\beta> 1$,
		$h\in C^\alpha([0, T]; \mathbb{R}) $ and $g\in C^\beta([0, T]; \mathbb{R})$. Then, for $(s,t)\in{\mathcal S}_{2}([0,T])$, the Young integral 
		$\int_{s}^{t} h(r) \, dg(r)$
		is defined as the limit of Riemann sums. That is,
		$$\int_{s}^{t} h(r) \, dg(r)=\lim_{|\pi_{s,t}|\to0}\sum_{i=0}^{m-1}h(r_i)\delta g(r_i,r_{i+1}),$$
		where $\pi_{s,t}=\{s=t_0<t_1,\ldots<,t_m=t\}$ is a partition of $[s,t]$.
		Furthermore, there exists a constant $c_{\alpha,\beta}>0$ such that
		\begin{equation*}
			\left| \int_{s}^{t} h(r) \, dg(r)  \right|
			\le
			|h(s)| |\delta g(s,t)|  +c_{\alpha,\beta} \|h\|_{\alpha,[s,t]} \|g\|_{\beta,[s,t]} |t-s|^{\alpha+\beta}  .
		\end{equation*}
	\end{prop}
	For the proof of this result, the reader can see \cite{LeonTindel} (Theorem 2.5) and references
	therein, or \cite{Leon} (Theorem 33).
	
	It is well-known that the paths of  fractional Brownian motion $B^H$  belong to the space $C^\gamma([0, T]; \mathbb{R})$, for any $\gamma<H$. Therefore, the integral in equation \eqref{eqi1} is interpreted as a
	pathwise integral (i.e., $\omega$ by $\omega$).
	
	On the other hand, we will apply the following result in the remaining of this article. Its proof 
	can be found in \cite{LiuT} (Lemma 2.5) and \cite{Garzon} (Lemma 2.2).
	\begin{lemma}
		\label{lemswing}
		Let $h$ be a function defined on $[\![ 0, T]\!]^{2}$ and $\mu>1$. Then the following inequalities hold true:
		
		\begin{enumerate}[(i)]
			\item Whenever $h(t_i, t_{i+1})=0$ for all $0\le i <n-1$, we have
			$$
			\left\|h\right\|_{\mu}\leq C_{\mu}\left\|\delta h\right\|_{\mu},
			$$
			with
			$$	\left\|h\right\|_{\mu}=\sup_{(u,v)\in\mathcal{S}_2([\![0, T ]\!])}\frac{|h(u,v)|}{|v-u|^\mu}
			\quad\text{and}\quad
			\left\|\delta h\right\|_{\mu}=\sup_{(s,u,v)\in\mathcal{S}_3([\![0, T ]\!])}
			\frac{|\delta h(s,u,v)|}{|v-s|^\mu}.$$
			\item In the general case where the quantities $h(t_i, t_{i+1})$ do not all vanish, we obtain:
			\begin{equation*}
				\left\|h\right\|_{\mu}\leq C_{\mu}\left\|\delta h\right\|_{\mu}
				+ \mathcal{M}_{\mu}(h),
				\quad\text{where}\quad
				\mathcal{M}(h) \equiv \sup\left\{ \frac{|h(t_i, t_{i+1})|}{|t_{i+1}-t_{i}|^{\mu}} ; \,  0\le i <n-1 \right\} .
			\end{equation*}
		\end{enumerate}
	\end{lemma}
	\subsection{Stochastic functional differential equations}
	Here, we establish the existence of a unique solution to the stochastic differential equation
	\begin{eqnarray}
		\label{eq1}
		dX(t)&=& f(\overline{X}_t)dB^H(t), \quad t\in (0, T] ,\\
		\overline{X}_0&=& \xi. \notag
	\end{eqnarray}
	Remember that $B^H$ is a fractional Brownian motion with Hurst parameter $H>1/2$ and that, for $t\in[0,T]$, $\overline{X}_t:[-\tau,0]\to\mathbb{R}$ is given by $\overline{X}_t(\cdot)=X(t+\cdot)$.
	
	In order to state the conditions on $\xi$ and the coefficient $f$, we need to introduce the following space. 
	
	Henceforth, we deal with a fixed $\tau>0$. For $0\le a_1<a_2$ and $\rho\in C^\mu([a_1-\tau,a_1];\mathbb{R})$,
	we introduce the complete metric space 
	$$C^\mu_{\rho, a_1, a_2}(\mathbb{R}):=\{\zeta\in C^\mu ([a_1 - \tau, a_2], \mathbb{R}): \zeta=\rho \ \text{on} \ [a_1-\tau, a_1]\}$$
	with the distance
	$$d_{\mu, a_1, a_2,\rho} (\zeta_1,\zeta_2) = \| \zeta_1 - \zeta_2 \|_{\mu, [a_1-\tau, a_2]}.$$
	To guarantee the existence of a solution to equation (\ref{eq1}), we suppose that the coefficient $f $ satisfies the following hypotheses:
	
	\begin{hypothesis}
		\label{H1}
		There exists a positive constant $M_1$ and $\lambda \in  (1/2, H)$ such that
		$\xi\in C^\lambda([-\tau,0];\mathbb{R})$ and $f:C^\lambda([-\tau,0];\mathbb{R})\to\mathbb{R}$ satisfies
		$$|f(\psi_2) - f(\psi_1)| \leq  M_1 \|\psi_2-\psi_1\|,\quad
		\text{uniformly in} \ \psi_1, \psi_2 \in C^{\lambda} ([-\tau, 0], \mathbb{R}).
		$$
		
	\end{hypothesis}

	\begin{hypothesis}
		\label{H2}
		Let $a = (a_1, a_2)$, $0 \leq a_1 < a_2 \leq T$ and $\lambda$ as in Hypothesis \ref{H1}. For any $N\ge 1$ and $\rho \in C^\lambda ([a_1 - \tau, a_1])$, there exists a positive constant $c_N$ such that
		$$\|\mathcal{U}(Z)-\mathcal{U}(W)\|_{\lambda, [a_1, a_2]} \le c_N \|Z-W\|_{\lambda, [a_1-\tau, a_1]},$$
		for $Z, W\in C^\lambda_{\rho, a_1, a_2}(\mathbb{R})$, satisfying
		$$\max\{\|Z\|_{\lambda, [a_1-\tau, a_2]}, \|W\|_{\lambda, [a_1-\tau, a_2]}\}\leq N.$$
		Here,
		\begin{equation*}
			\label{eqdefUZ}
			\mathcal{U}(Z)(s)=f(\overline{Z}_s), \quad s\in[a_1, a_2].
		\end{equation*}
	\end{hypothesis}
	\begin{remark}
		\label{remarklineal}
		Note that Hypothesis \ref{H1} implies that there exist a positive constant $M_2$ such that
		$$|f(\psi)| \leq  M_2 (1+\left\|\psi\right\|_{[-\tau, 0]}),\quad
		\text{for all}\ \psi \in C^{\lambda} ([-\tau, 0], \mathbb{R}).$$
	\end{remark}
	
	Now, we can state an existence and uniqueness result for the solution to equation \eqref{eq1}.
	\begin{theorem} 
		\label{teo solution}
		Under Hypotheses \ref{H1} and \ref{H2}, equation \eqref{eq1} has a unique solution $X$ in $C^\lambda_{\xi, 0, T}(\mathbb{R})$.
		Furthermore, for all $(s, t)\in \mathcal S_{2}([0,T])$,
		\begin{equation}
			\label{HolderX}
			\delta X({s, t})=f(\overline{X}_s)\delta B^H(s, t) + R^{X}(s,t),
		\end{equation}
		where $R^X$ satisfies
		\begin{equation}
			\label{HolderX2}
			|R^{X}(s,t)|\leq C_X |t-s|^{2\lambda},\quad
			\text{for some positive constant}\  C_X.
		\end{equation}

	\end{theorem}
	\begin{remark}
		Le\'on and Tindel \cite{LeonTindel} have shown that equation \eqref{eq1} has a unique solution in 
		$C^\lambda_{\xi, 0, T}(\mathbb{R})$ assuming that, in addition to Hypotheses \ref{H1} and \ref{H2}, there exists
		a constant $M>0$ such that
		$$|f(\zeta)|\le M, \quad\text{for every } \zeta\in C^{\lambda}([-\tau,0];\mathbb{R}).$$
		Moreover, this existence and uniqueness result does not depend on the Fréchet differentiability of $f$ unlike reference \cite{Boufoussi}.
	\end{remark}
	\begin{proof}
		Here, we consider $\gamma\in(\lambda,H)$ and $\omega\in\Omega$ such that $B^H(\omega,\cdot)$ is in $C^\gamma([0,T];\mathbb{R})$. As usual, we omit $\omega$ in the following calculations.
		
		In the proof of Theorem 3.2 in Le\'on and Tindel \cite{LeonTindel}, the authors have chosen $\varepsilon\in(0,T)$ and introduced the map $\Gamma: \mathcal{C}_{\xi, 0, \varepsilon}^\lambda\left(\mathbb{R}\right) \rightarrow \mathcal{C}_{\xi, 0, \varepsilon}^\lambda\left(\mathbb{R}\right)$ given as follows: if $z \in \mathcal{C}_{\xi, 0, \varepsilon}^\lambda\left(\mathbb{R}\right)$, then $\Gamma(z)=\hat{z}$, where $\hat{z}_t=\xi_t$ for $t \in[-\tau, 0]$, and $\delta \hat{z}(s, t)=\int_s^tY(u) d B^H(u),$ where $Y(u)=f\left(\overline{z}_u\right)$, for $ s, t \in[0, \varepsilon]$.  In consequence, using Proposition \ref{prop:young-integral} and
		proceeding as in the proof of Theorem 3.2 in \cite{LeonTindel}, we obtain
		\begin{equation}
			\label{eqA3}
			\| \delta\hat{z}\|_{\lambda,[0, \varepsilon]} \leq\|Y\|_{\infty,[0, \varepsilon]}\|B^H\|_\gamma \varepsilon^{\gamma-\lambda}+c_{\gamma, \lambda}\|Y\|_{\lambda,[0, \varepsilon]}\|B^H\|_\gamma \varepsilon^\gamma.
		\end{equation}
		Moreover, Hypothesis \ref{H1} implies  that, for $u \in [0,\varepsilon]$,
		$$
		|f\left(\overline{z}_u\right)| \leq |f\left(\overline{z}_u\right)-f\left(\overline{z}_0\right)|+|f\left(\overline{z}_0\right)| \leq M_1 \sup _{\theta \in[-\tau, 0]}\left|z_{u+\theta}-z_{\theta}\right|+|f\left(\xi\right)|.
		$$
		Consequently, we are able to write
		\[\|Y\|_{\infty,[0,\varepsilon]}\leq M_1 (T+\tau)^\lambda \|z\|_{\lambda,[-\tau, \varepsilon]} + C_\xi,\] 
		where $C_\xi :=|f(\xi)|$. Now, suppose that $\|z\|_{\lambda,[0,\varepsilon]}\le N_1$. Then, \cite{LeonTindel} (Lemma 3.1) and \eqref{eqA3} 
		yield
		\begin{equation*}
			\| \hat{z}\|_{\lambda,[0, \varepsilon]} \leq M_1 (T+\tau)^\lambda\|B^H\|_\gamma \varepsilon^{\gamma-\lambda}N_1+C_\xi \|B^H\|_\gamma\varepsilon^{\gamma-\lambda}+M_1c_{\gamma, \lambda} N_1 \varepsilon^\gamma\|B^H\|_\gamma.
		\end{equation*}
		Therefore, proceeding as in \cite{LeonTindel} (proof of Theorem 3.2), equation \eqref{eq1} has a unique solution in $C^\lambda_{\xi, 0, T}(\mathbb{R})$ 
		by chosen  
		$$\varepsilon=\left[6 M_1 (T+\tau)^\gamma\|B^H\|_\gamma\right]^{-1 / \gamma-\lambda} \wedge \left[6M_1 c_{\gamma, \lambda}\|B^H\|_\gamma\right]^{-1 / \gamma} \wedge 1$$ 
		
		and  $N_1 \geq \sup\{6 C_\xi\|B^H\|_\gamma,2\|\xi\|_{\lambda,[\tau,0]}\}$
		instead of those given in the proof of Theorem 3.2 in \cite{LeonTindel}.
		It is worth mentioning that we need to do these changes because now we do not have that there exists some constant $M>0$ such that $|f(\zeta)|\le M$, for every $\zeta\in C^{\lambda}([-\tau,0];\mathbb{R}).$	
		
		Finally, we deal with \eqref{HolderX} and \eqref{HolderX2}. Let $(s, t)\in \mathcal S_{2}([0,T])$,
		then
		\begin{align*}
			\delta X({s,t})&=X(t)-X(s)= \int_s^t f(\overline{X}_u)dB^H_u \notag \\
			&= \int_s^t f(\overline{X}_s)dB^H(u)  + \int_s^t [f(\overline{X}_u)- f(\overline{X}_s)]dB^H(u)   \notag \\
			&= f(\overline{X}_s)\delta B^H(s,t) + R^{X}(s,t),
		\end{align*}
		where $R^{X}(s,t):= \int_s^t F(u)dB^H(u)$ and $F(u):= f(\overline{X}_u)- f(\overline{X}_s)$.
		Hence, from Proposition \ref{prop:young-integral}, we get
		\begin{align}
			\label{cota RX0}
			|R^{X}(s,t)| &\leq   C_{\lambda} \left(|F(s)|\, |\delta B^H(s, t)| + \|F\|_{\lambda, [s,t]}\, \| B^H\|_{\lambda, [s,t]} |t-s|^{2\lambda}\right) \notag \\
			&=  C_{\lambda}  \|F\|_{\lambda, [s,t]}\, \| B^H\|_{\lambda, [s,t]} |t-s|^{2\lambda}.
		\end{align}
		
		We now estimate the $\lambda$-H\"older norm of $F$: the definition of $F$ and Hypothesis \ref{H1}
		allow us to write
		\begin{align*}
			\|F\|_{\lambda, [s,t]}&= \sup_{(u,v)\in \mathcal{S}_2([s,t])} \frac{|\delta F(u,v)|}{|v-u|^\lambda} 
			= \sup_{(u,v)\in \mathcal{S}_2([s,t])} \frac{|f(\overline{X}_v)- f(\overline{X}_u)|}{|v-u|^\lambda} \notag \\
			&\leq \sup_{(u,v)\in \mathcal{S}_2([s,t])} M_1 \sup_{\theta \in [-\tau, 0]}\frac{|{X}(v+\theta)- {X}(u+\theta)|}{|v-u|^\lambda} \notag \\
			&\leq \sup_{(u,v)\in \mathcal{S}_2([s,t])} M_1 \sup_{\theta \in [-\tau, 0]} \|X\|_\lambda\frac{|v- u|^{\lambda}}{|v-u|^\lambda} =
			M_1\|X\|_\lambda.
		\end{align*}
		Thus,  \eqref{cota RX0} leads to
		$$	|R^{X}(s,t)|  \leq C_{\lambda}  M_1\|X\|_\lambda \, \| B^H\|_{\lambda, [0,T]} |t-s|^{2\lambda} 
		= C_X |t-s|^{2\lambda}.$$
		Therefore, the proof is complete.\qed
	\end{proof}
	\section{Euler Approximation}\label{Euler}
	Trhoughout this section, we fix $\gamma\in(1/2,H)$. In consequence, we can assume that $B^H$ is  $\gamma$-H\"older continuous on $[0,T]$. It means, we have fixed an $\omega\in\Omega$ such that
	$B^H(\omega,\cdot)\in C^\gamma([0,T];\mathbb{R})$. Moreover, we suppose that Hypotheses \ref{H1} and \ref{H2} are satisfied.

	The numerical method for equation \eqref{eq1} implemented in this paper is only 
	studied in $[0, \tau]$ because, in  the remaining interval $[\tau,T]$, it is handled recursively. That means, if the Euler scheme for \eqref{eq1} is solved on the interval $[(k-1)\tau, k\tau]$, then the following step is to consider the equation
	\begin{eqnarray*}
		%\label{eq1-k}
		dX_{k}(t)&=& f(\overline{(X_k)}_t)dB^H(t), \quad t\in [k\tau, (k+1)\tau], \\
		X_k(k\tau)&=& X_{k-1}, \notag
	\end{eqnarray*}
	where $X_{k-1}\in C^\lambda([(k-1)\tau, k\tau]; \mathbb{R})$ is the solution of equation \eqref{eq1} on $[(k-1)\tau, k\tau]$. Taking into account this argument, now
	we introduce an Euler-Maruyama scheme to approximate the solution to equation \eqref{eq1}.      Thus, we  can assume that $T=\tau$ without loss of generality. 
	Therefore, the time step size for our scheme is
  	\begin{equation}
		\label{defdelta}
		\Delta=\Delta_n:= \frac{\tau}{n}, \quad n\in \mathbb{N}, \ n> \tau.
	\end{equation}
	
	\ 
	
	Let $\left[\![ -\tau, T\right]\!]_n=\{-\tau=t_{-n} <\cdots < 0=t_0 < t_1 < \cdots < t_{n} =T\}$ be a partition of the time interval $[-\tau, T]$ with $t_{k}=k\Delta$ for $k=-n, \cdots, 0, \cdots n$. 
	If $s, t \in [-\tau, T]$ with $s<t$, we denote
	$$\left[\![ s, t\right]\!]_n := [s, t]\cap \left[\![ -\tau, T\right]\!]_n.$$

	The Euler-Maruyama scheme for equation \eqref{eq1}  is defined continuously as
		\begin{equation}
		\label{euler}
		\begin{cases}
			X^{(n)}(t
			)= \xi(t), & \ t \in [-\tau,0]\\ 
			X^{(n)}(t)=X^{(n)}(k\Delta)+f(\overline{X}^{(n)}_{k\Delta})\delta B^H(k\Delta,  t), & \ t \in [k\Delta, (k+1)\Delta], \, k= 0, 1,\cdots, n-1,
		\end{cases}    
	\end{equation}
     %%%%%%%%%% quite 	Remember that 
     where
	$\overline{X}^{(n)}_{t}$ belongs to $C^\lambda([-\tau,0], \mathbb{R})$ and it is given by
    	\begin{align*}
		%	\label{eq6.1}
		\overline{X}^{(n)}_{t}(\theta):=\{ X^{(n)}(t+\theta) :  \theta \in  [-\tau,0]\}, \quad t\in [0, T].
	\end{align*}
	Note that we are able to  equivalently write the Euler scheme as:
	
	$$X^{(n)}{(t_{i+1})}=X^{(n)}(t_{i})+ f(\overline{X}^{(n)}_{t_{i}})\delta B^H({t_i, t_{i+1}}), \quad i=0, \cdots n-1.$$
	That is,
	\begin{equation}
		\label{eeuler1}
		\delta X^{(n)}({t_i, t_{i+1}})= f(\overline{X}^{(n)}_{t_i})\delta B^H({t_i, t_{i+1}}).		
	\end{equation}
	This relation can be extended to any pair of points $(s,t)$ in ${\mathcal S}_{2}([\![ 0, T ]\!]_{n})$, with $s< t$, as follows:
	\begin{equation}
		\label{deltaeuler}
		\delta X^{(n)}({s, t})= f(\overline{X}^{(n)}_{s})\delta B^H({s, t}) + R^n({s,t}),		
	\end{equation}
	where we have
	\begin{equation}
		\label{resto}
		R^n({s, t}):=	\delta X^{(n)}({s, t}) - f(\overline{X}^{(n)}_{s})\delta B^H({s, t}),
	\end{equation}
	with $R^n({t_i, t_{i+1}})=0$.
	
	\ 
	
	Remember that we are assuming  $T=\tau$ throughout the remainder of the paper.

	\begin{lemma}\label{lem1} Assume that Hypothesis \ref{H1} holds. Then,
		for $(s, u, t) \in {\mathcal S}_{3}([\![ 0, T]\!]_n)$, we have
		\begin{equation*}%\label{e1}
			|\delta R^n({s,u,t})|
			\le
			M_1 \|\delta \overline{X}^{(n)}(s, u)\|  \left\|B^H\right\|_{\lambda}|t-s|^{\lambda} ,
		\end{equation*}
	\end{lemma}
	where $M_1$ is the constant in Hypothesis \ref{H1}.
	\begin{proof}
		Applying the operator $\delta$ on both sides of \eqref{resto}, and using that $\delta\delta X^{(n)}=0$, $\delta\delta B^H=0$ and Hypothesis \ref{H1}, we obtain
		\begin{align*}%\label{b1}
			|\delta R^n({s,u,t})|&=	|\delta\delta X^{(n)}({s,u,t}) - \delta[f(\overline{X}^{(n)}_s)\delta B^H]({s,u,t})| \notag\\
			&=|\delta f(\overline{X}^{(n)})({s, u})\delta B^H({u, t}) - f(\overline{X}^{(n)}_s)\delta(\delta B^H)({s, u, t})|  \notag\\
			&=|\delta f(\overline{X}^{(n)})({s, u})\delta B^H({u, t})|
			\leq M_1 \|\delta \overline{X}^{(n)}({s, u})\|  \|B^H\|_{\lambda}|t-s|^{\lambda}.
		\end{align*}\qed
	\end{proof}
	
	Now, we establish the following auxiliary result.
	\begin{lemma}
		\label{lema2} Let Hypothesis \ref{H1} be satisfied. Then,
		for all $s, t \in [0,T]$,
		\begin{equation}\label{eq:bnd-holder-Yn}
			|\delta X^{(n)}({s, t})| \leq \hat{c} \, |t-s|^{\lambda}
		\end{equation}
		and 
		\begin{equation}\label{eq:bnd-holder-Yn2}
			\|\delta \overline{X}^{(n)}({s, t})\| \leq \hat{c} \, |t-s|^{\lambda},
			% acá || es la norma sup
		\end{equation}
		where $\hat{c}$ is a constant such that $\hat{c}\le c_{1} \exp(c_{2}(1+\|B^H\|_{\lambda}^{1/\lambda}))$ for some  constants $c_{1},c_{2}>0$.
		
	\end{lemma}
	
	\begin{proof}
		We divide the proof into several steps. We begin by proving, via induction, that for all $l = 1, 2, \cdots n$ and $s, t \in   [0,T]$ with $0 \leq s < t \leq t_l$ we have:
		\begin{equation}
			\label{eqinduccion}
			|\delta {X}^{(n)}({s, t})| \leq C_1 \, |t-s|^{\lambda} \ \text{and} \  \|\delta \overline{X}^{(n)}({s, t})\| \leq C_1 \, |t-s|^{\lambda}.
		\end{equation}

		\textbf{Step 1}: Here, we deal with   $l = 1$ and  $s<t$. 
		
		By Remark \ref{remarklineal}, \eqref{euler}  and the H\"older property of $B^H$, we have
		\begin{align}
			\label{eq311}
			|\delta X^{(n)}({s, t})| &=	|f( \overline{X}^{(n)}_{0})| |\delta B^H({s, t})| \notag\\
			&\leq M_2 (1+ \|\overline{X}^{(n)}_{0}\|)  \|B^H\|_{\lambda}|t-s|^{\lambda}.
		\end{align}
		
		For \eqref{eq:bnd-holder-Yn2}, we write
		\begin{align}
			\label{normadeltaXb}
			\|\delta \overline{X}^{(n)}({s, t})\|  &= \sup_{\theta\in [-\tau, 0]} |\overline{X}^{(n)}_t(\theta)- \overline{X}^{(n)}_s(\theta)| 
			=\sup_{\theta\in [-\tau, 0]} |\delta {X}^{(n)}(s+\theta, t+ \theta)|.           
		\end{align}	
		We now derive an upper bound for $|\delta X^{(n)}(s+\theta, t+\theta)|$ for $\theta \in [-\tau,0]$. To do so, we consider three cases:     
		
		\textbf{Case 1. }  Here, $-\tau\leq \theta \leq -t$.
		
		In this case, $s+\theta$ and $t+\theta$ are less than or equal to zero. Thus, Hypothesis \ref{H1} leads us to
		$$|\delta X^{(n)}(s+\theta, t+\theta)|=|\xi(t+\theta)- \xi(s+ \theta)|\leq \|\xi\|_\lambda|t-s|^{\lambda}.$$ 
		\textbf{Case 2. }  Now $-s\leq \theta \leq 0$.
		
		Now, $s+\theta$ and $t+\theta$ are non-negative and less than $t_1$. Hence, by \eqref{eq311},
		$$|\delta X^{(n)}(s+\theta, t+\theta)|\leq M_2 (1+ \|\overline{X}^{(n)}_{0}\|)  \|B^H\|_{\lambda}|t-s|^{\lambda}.$$
		\textbf{Case 3. } Finally, $-t<\theta< -s$.
		
		We have that $s-t< \theta +s <0$ and $0< \theta + t< t-s$. Using 
		Hypothesis \ref{H1} again, together with  \eqref{eq311}, we obtain
		\begin{align*}
			|\delta X^{(n)}(s+\theta, t+\theta)|&\leq |\delta X^{(n)}(s+\theta, 0)| + |\delta X^{(n)}(0, t+\theta)|\\
			&\leq \|\xi\|_\lambda|s+\theta|^{\lambda} +   M_2 (1+ \|\overline{X}^{(n)}_{0}\|)  \|B^H\|_{\lambda}|t+\theta|^{\lambda}\\
			&\leq (\|\xi\|_\lambda+   M_2 (1+ \|\overline{X}^{(n)}_{0}\|)  \|B^H\|_{\lambda})|t-s|^{\lambda}.
		\end{align*}
		
		Therefore, from  \eqref{normadeltaXb}, we obtain 
		\begin{align*}%\label{eq311-b}
			\|\delta \overline{X}^{(n)}({s, t})\| 
			\leq & (\|\xi\|_\lambda  + M_2 (1+ \|\overline{X}^{(n)}_{0}\|)  \|B^H\|_{\lambda})|t-s|^{\lambda}.
		\end{align*}
		Taking $C_1\ge \|\xi\|_\lambda  + M_2 (1+ \|\overline{X}^{(n)}_{0}\|)  \|B^H\|_{\lambda}$,  we have the result for $l=1$. That is, \eqref{eqinduccion} is true for  $l=1$.

		\textbf{Step 2: Upper bound for $\mathbf{R^n}$}.
		The induction assumption is that (\ref{eqinduccion}) is true for   $s,t \in [0,T]$ with $0\leq s < t \leq t_{l}$.
		
		Let $s, t \in \left[\![ 0, \tau\right]\!]_n$ with $0\leq s < t \leq t_{l+1}$. 
		To estimate $|\delta X^{(n)}{(s, t)}|$, from \eqref{deltaeuler} and Lemma \ref{lemswing}, we first need to bound $|\delta R^n({s,u,t})|$. 
		So,  we first focus on the case $0\leq s < t=t_{l+1}$.
		For $0\leq s < u < t = t_{l+1}$,  $u \in \left[\![ 0, \tau\right]\!]_n$, with the help of Lemma \ref{lem1} and induction hypothesis,  we have that
		\begin{align*}%	\label{eqE5}
			|\delta R^n({s,u,t})| &\leq M_1 \|\delta \overline{X}^{(n)}({s,u})\|  \left\|B^H\right\|_{\lambda}|t-s|^{\lambda} \notag\\
			&\leq M_1C_1|u-s|^\lambda\left\|B^H\right\|_{\lambda}|t-s|^{\lambda}\notag\\
			&\leq M_1C_1\left\|B^H\right\|_{\lambda}|t-s|^{2\lambda}.
		\end{align*}
		Note that this inequality remains valid if
		$t<t_{l+1}$. Applying Lemma \ref{lemswing} over $\left[\![ 0, t_{l+1}\right]\!]_n$, we obtain:
		$$
		\| R^n\|_{[s,t], \lambda}\leq C_\lambda M_1C_1\left\|B^H\right\|_{\lambda}|t-s|^{\lambda}.
		$$
		Hence, for $0\le  u \le v  \le t_{l+1}$, $u,v \in \left[\![ 0, \tau\right]\!]_n$,
		\begin{equation}\label{cotaR}
			| R^n({u,v})|  \leq C_\lambda M_1C_1\left\|B^H\right\|_{\lambda}|v-u|^{2 \lambda}.
		\end{equation}
		
		%%%%%%%%%%%%%%%%%%%%%%%%%%%%%%%%%%%%%%%%%%%%%%%%%%%%%%%
		
		\textbf{Step 3: Local upper bound for $\mathbf{\delta X^{n}}$ and $\mathbf{\delta \overline{X}^{n}}$} 
		
		We will establish a uniform bound of $X^{(n)}({t})$ and $\overline{X}^{(n)}_{t}$. First, for $t\in [\![ 0, t_{l+1}]\!]$, from Remark \ref{remarklineal}, \eqref{deltaeuler} and  \eqref{cotaR}, we obtain 
		\begin{align}
			\label{DXst-0}
			|\delta X^{(n)}{(0, t)}|  &\leq  |f(\overline{X}^{(n)}_0)\|\delta B^H({0, t})| + |R^n{(0,t)}| \nonumber \\ 
			&\leq  M_2 ( 1 + \|\overline{X}^{(n)}_0\| ) |\delta B^H({0, t})| + |R^n{(0,t)}| \nonumber \\
			&\leq  M_2 ( 1 +\|\overline{X}^{(n)}_0\|) \left\|B^H\right\|_{\lambda}|t|^\lambda + C_\lambda M_1C_1\left\|B^H\right\|_{\lambda}|t|^{2 \lambda}\nonumber \\
			&=   \left(M_2(1 + \|\overline{X}^{(n)}_0\|) + C_\lambda M_1C_1|t|^{ \lambda}\right)\left\|B^H\right\|_{\lambda}|t|^\lambda.
		\end{align}
		
		From \eqref{DXst-0}, for $t\in [\![ 0, t_{l+1}]\!]$, we get 
		\begin{align}
			\label{cotaXnt}
			| X^{(n)}(t)|  & = |X^{(n)}(0) + \delta X^{(n)}(0, t)|\notag\\
			&\leq  |X^{(n)}(0)| +\left(M_2 ( 1 + \|\overline{X}^{(n)}_0\|)+ C_\lambda M_1C_1|t|^{ \lambda} \right)\left\|B^H\right\|_{\lambda}|t|^\lambda\notag\\
			& \leq  |X^{(n)}(0)| +  \|\overline{X}^{(n)}_0\|M_2\left\|B^H\right\|_{\lambda}t^\lambda+ J'_{2t}\notag\\
			& \leq  J'_{1t}\|\overline{X}^{(n)}_0\| +J'_{2t}, 
		\end{align}
		where 
		\begin{equation*}
			%	\label{J's}
			J'_{1t}=1+ M_2\left\|B^H\right\|_{\lambda}t^\lambda \quad  \text{and} \quad  J'_{2t}=\left(M_2 +  C_\lambda M_1C_1|t|^{ \lambda} \right)\left\|B^H\right\|_{\lambda}t^\lambda.
		\end{equation*}
		For $s\in [0,t_{l+1})$, let $t_j$ be the closest member of the partition to $s$ such that $t_j \leq s$. Then by Remark \ref{remarklineal} and the induction hypothesis,
		\begin{align}
			\label{cotaXns}
			|X^{(n)}(s)|&\leq|X^{(n)}(t_j)|+|X^{(n)}(s)-X^{(n)}(t_j)| \notag\\
			&= |X^{(n)}(t_j)| + |f(\overline{X}^{(n)}_{t_j})\delta B^H({t_j, s})| \notag\\
			&\leq  J'_{1t_j}\|\overline{X}^{(n)}_0\| +  J'_{2t_j} + M_2(1+\|\overline{X}^{(n)}_{t_j}\|)\left\|B^H\right\|_{\lambda}|s-t_j|^{\lambda} \notag\\
			&\leq  J'_{1s}\|\overline{X}^{(n)}_0\| +  J'_{2s} + M_2 \left(1+\|\delta\overline{X}^{(n)}(0, t_j)\|+\|\overline{X}^{(n)}_{0}\|\right)\left\|B^H\right\|_{\lambda}|s-t_j|^{\lambda} \notag\\
			&\leq  J'_{1s} \|\overline{X}^{(n)}_0\| + J'_{2s} + M_2 \left(1+C_1 |t_j|^{\lambda}+\|\overline{X}^{(n)}_0\|\right)\left\|B^H\right\|_{\lambda}|s-t_j|^{\lambda} \notag\\
			&\leq  J'_{1s} \|\overline{X}^{(n)}_0\| +  J'_{2s} + M_2 \left(1+C_1 |s|^{\lambda}+\|\overline{X}^{(n)}_0\|\right)\left\|B^H\right\|_{\lambda}|s|^{\lambda}.
		\end{align}
		
		Thus, based on the bound of $|\delta X^{(n)}{(0, t)}|$ given in \eqref{DXst-0}, we get,  from \eqref{cotaXnt}  and \eqref{cotaXns}, for all $s\in [0,t_{l+1}]$, that
		\begin{align}
			\label{DXs-01}
			|X^{(n)}(s)|&\leq J'_{1s}\|\overline{X}^{(n)}_0\| +J'_{2s} + M_2 \left(1+C_1 |s|^{\lambda}+\|\overline{X}^{(n)}_0\|\right)\left\|B^H\right\|_{\lambda}|s|^{\lambda} \notag \\
			&\leq   J_{1s}\|\overline{X}^{(n)}_0\| +J_{2s} \equiv  K(\overline{X}^{(n)}_0, C_1, f, B^H, s), 
		\end{align}
		where we define
		\begin{equation}
			\label{Js}
			J_{1s}=1+ 2M_2\left\|B^H\right\|_{\lambda}s^\lambda \quad  \text{and} \quad J_{2s}=\left(2M_2 +  C_\lambda M_1C_1|s|^{ \lambda}+M_2C_1|s|^{ \lambda} \right)\left\|B^H\right\|_{\lambda}s^\lambda.
		\end{equation}
		Note that
		\begin{equation}
			\label{JsJt}
			J_{1s}\leq J_{1t} \quad \text{and} \quad J_{2s}\leq J_{2t}, \quad \text{for} \ s\leq t.
		\end{equation}
		Now,  we will study  $\| \overline{X}^{(n)}_s\|=\sup_{\theta\in [-\tau, 0]} |{X}^{(n)}(s+\theta)|.$ This analysis is divided in two cases.
		
		\textbf{Case 1.} $\theta\in [-\tau, -s]$, then $s+\theta \leq 0$ and by \eqref{Js}
		$$|{X}^{(n)}(s+\theta)|=|\xi(s+\theta)|\leq \|\overline{X}^{(n)}_0\| \leq  J_{1s}\|\overline{X}^{(n)}_0\| +J_{2s}.$$
		\textbf{Case 2.} $\theta\in (-s, 0]$, then $0<s+\theta \leq s$ and from \eqref{DXs-01}
		$$|{X}^{(n)}(s+\theta)|\leq J_{1s}\|\overline{X}^{(n)}_0\| +J_{2s} .$$
		Hence, for $J_{1s}$ and $J_{2s}$ given in \eqref{Js}
		\begin{align}
			\label{DXs-02}
			\| \overline{X}^{(n)}_s\|  &\leq  J_{1s}\|\overline{X}^{(n)}_0\| +J_{2s}.
		\end{align}

		We continue our analysis by studying the increment $\delta X^n{(s, t)}$.  First, we take  $s,t \in \left[\![ 0, t_{l+1} \right]\!]_n$. From Remark \ref{remarklineal}, \eqref{deltaeuler}   and \eqref{DXs-02},
		\begin{align}
			\label{DXst}
			|\delta X^{(n)}{(s, t)}|  &\leq  |f(\overline{X}^{(n)}_s)||\delta B^H({s, t})| + |R^n{(s,t)}| \nonumber \\ 
			&\leq  M_2 \left( 1 + \|\overline{X}^{(n)}_s\| \right) |\delta B^H({s, t})| + |R^n{(s,t)}| \nonumber \\
			&\leq  M_2 \left( 1 +  J_{1s}\|\overline{X}^{(n)}_0\| +J_{2s} \right)  |\delta B^H({s, t})|+ |R^n{(s,t)}|.
		\end{align}
		
		Secondly, let $t_i \leq s < t \leq t_{i+1}, \, i=\{0, \hdots, l\}$. From Remark \ref{remarklineal}, \eqref{DXs-01}, \eqref{JsJt} and \eqref{DXs-02},
		\begin{align} \label{DXsl}
			|\delta X^{(n)}{(s, t)}| &= |f(\overline{X}^{(n)}_{t_i})||\delta B^H({s, t})|\notag\\
			&\leq M_2 \left( 1 + \|\overline{X}^{(n)}_{t_i}\| \right) |\delta B^H({s, t})|\notag\\
			&\leq M_2 \left( 1 +  J_{1t_i}\|\overline{X}^{(n)}_0\| +J_{2t_i} \right) |\delta B^H({s ,t})|\notag\\
			&\leq M_2 \left( 1 +  J_{1s}\|\overline{X}^{(n)}_0\| +J_{2s} \right) |\delta B^H({s,t})|.
		\end{align}
		Finally, let $s \leq t_i < t_j \leq t$, where $t_i,t_j \in \left[\![ 0, t_{l+1} \right]\!]_n$, and we assume that $t_i$ is the smallest member of the partition bigger than $s$, while $t_j$ is the biggest member of the partition smaller than $t$. Then, by \eqref{DXst} and \eqref{DXsl}
		
		\begin{align*}
			&|\delta X^{(n)}{(s, t)}| \notag \\
			\leq&
			|\delta X^{(n)}{(s, t_i)}|+|\delta X^{(n)}{(t_i, t_j)}|+|\delta X^{(n)}{(t_j, t)}|\notag\\
			\leq&  M_2 \left( 1 + J_{1s}\|\overline{X}^{(n)}_0\| +J_{2s} \right) |\delta B^H({s,t_i})|  + M_2 \left( 1 +  J_{1t_i}\|\overline{X}^{(n)}_0\| +J_{2t_i} \right)  |\delta B^H({t_i,t_j})|+ |R^n{(t_i,t_j)}|\notag\\ 
			&+ M_2 \left( 1 + J_{1t_j}\|\overline{X}^{(n)}_0\| +J_{2t_j} \right) |\delta B^H({t_j,t})|.
		\end{align*}
		Using the bounds  \eqref{cotaR} and \eqref{JsJt}
		\begin{align*}
			|\delta X^{(n)}{(s, t)}| \leq&  M_2 \left( 1 + J_{1t}\|\overline{X}^{(n)}_0\| +J_{2t} \right)\left\|B^H\right\|_{\lambda}|s-t_i|^\lambda \notag \\
			&+ M_2 \left( 1 +  J_{1t}\|\overline{X}^{(n)}_0\| +J_{2t} \right)  \left\|B^H\right\|_{\lambda}|t_i-t_j|^\lambda+  C_\lambda M_1C_1\left\|B^H\right\|_{\lambda}|t_i-t_j|^{2 \lambda}\notag\\ 
			&+ M_2 \left( 1 + J_{1t}\|\overline{X}^{(n)}_0\| +J_{2t} \right) \left\|B^H\right\|_{\lambda}|t-t_j|^{\lambda}.
		\end{align*}

		As a consequence, for all $s,t \in [0,t_{l+1}], \, s<t$, we obtain the estimate
		\begin{align}\label{estSTGen}
			|\delta X^{(n)}{(s, t)}| &\leq 3\left\|B^H\right\|_{\lambda}|t-s|^\lambda \left[M_2 \left( 1 + J_{1t}\|\overline{X}^{(n)}_0\| +J_{2t} \right)  + C_\lambda M_1C_1|t-s|^{\lambda} \right].
		\end{align}

		%%%%%%%%%%%%%%%%%%%%%%%%%
		Now, we study the increment $\delta \overline{X}^n{(s, t)}$ for $s,t \in [0,t_{l+1}]$ with $s < t$.  That is, the analysis of
		\begin{equation}
			\label{eqE1}
			\|\delta \overline{X}^{(n)}({s, t})\|		
			= \sup_{\theta \in [-\tau,0]}  \{|\delta X^{(n)}(s+\theta, t+\theta)|\}.
		\end{equation}
		
		So, we need to  give an upper bound of $|\delta X^{(n)}(s+\theta, t+\theta)|$ for $\theta \in [-\tau,0]$.
		
		If $\theta \leq -t$, then $s+\theta \leq 0$ and $t+\theta\leq 0$.  In this case, using that $\xi\in C^\lambda([-\tau, 0]; \mathbb{R})$, we have
		\begin{equation}
			\label{cotaAmk}
			|\delta X^{(n)}(s+\theta, t+\theta)|  =  |\xi(t+\theta)-\xi(s+\theta)| 
			\leq   \|\xi\|_\lambda|t-s|^{\lambda}.
		\end{equation}
		If $-s \leq  \theta \leq 0$,  then $s+\theta$ and $t+\theta$ are non negative and lower or equal than $t_{l+1}$, therefore by \eqref{estSTGen}
		\begin{equation}
			\label{cotaCmk}
			|\delta X^{(n)}(s+\theta, t+\theta)|  \leq
			3\left\|B^H\right\|_{\lambda}|t-s|^\lambda \left[M_2 \left( 1 + J_{1t}\|\overline{X}^{(n)}_0\| +J_{2t} \right)  + C_\lambda M_1C_1|t-s|^{\lambda} \right].
		\end{equation}
		For $-t< \theta <-s$, then $0 < t+\theta< t-s$ and $s-t < s+\theta < 0$. From \eqref{estSTGen} and  calculations analogous to those previously outlined 
		\begin{align}
			\label{cotaBmk}
			&|\delta X^{(n)}(s+\theta, t+\theta)|\notag\\
			\leq & |\delta X^{(n)}(s+\theta, 0)| + |\delta X^{(n)}(0, t+\theta)|\notag\\
			\leq & \|\xi\|_\lambda|s+\theta|^{\lambda}  + 3\left\|B^H\right\|_{\lambda}|t+\theta|^\lambda \left[M_2 \left( 1 + J_{1t+\theta}\|\overline{X}^{(n)}_0\| +J_{2t+\theta} \right)  +  C_\lambda M_1C_1|t+\theta|^{\lambda} \right]\notag \\
			\leq & \|\xi\|_\lambda|t-s|^{\lambda}  + 3\left\|B^H\right\|_{\lambda}|t-s|^\lambda \left[M_2 \left( 1 + J_{1t}\|\overline{X}^{(n)}_0\| +J_{2t} \right) + C_\lambda M_1C_1|t-s|^{\lambda} \right].
		\end{align}
		By \eqref{eqE1}, \eqref{cotaAmk}, \eqref{cotaCmk} and \eqref{cotaBmk}, for all $s,t \in [0,t_{l+1}]$; and by replacing the values of $J_{1t}$ and $J_{2t}$ given in \eqref{JsJt}, we get
		\begin{align}
			\label{eqE1-1}
			\|\delta \overline{X}^{(n)}({s, t})\|
			\leq & \left[\|\xi\|_\lambda  + 3\left\|B^H\right\|_{\lambda} \left[M_2 \left( 1 + J_{1t}\|\overline{X}^{(n)}_0\| +J_{2t} \right) + C_\lambda M_1C_1|t-s|^{\lambda} \right]\right]|t-s|^\lambda\notag \\
			\leq & \left[\|\xi\|_\lambda  + 3\left\|B^H\right\|_{\lambda} \left[M_2 \left( 1 + \|\overline{X}^{(n)}_0\|+2M_2 \left\|B^H\right\|_{\lambda}|t|^\lambda( 1 + \|\overline{X}^{(n)}_0\|\right)  \right.\right.\notag\\
			& \hspace{1.5cm} \left. \left.+ C_1|t|^{ \lambda}\big(M_2C_\lambda M_1\left\|B^H\right\|_{\lambda}|t|^{\lambda} + M_2^2 \left\|B^H\right\|_{\lambda}|t|^\lambda  + C_\lambda M_1\big) \right]\right]|t-s|^\lambda.
		\end{align}
		
		Since the bound for $|\delta X^{(n)}{(s, t)}| $  given in \eqref{estSTGen} is less than the one given for $\|\delta \overline{X}^{(n)}({s, u})\|$ in \eqref{eqE1-1}, we consider the largest
		one. Then 
		\begin{align}
			\label{DXst-2}
			|\delta X^{(n)}{(s, t)}|
			\leq & \left[\|\xi\|_\lambda  + 3\left\|B^H\right\|_{\lambda} \left[M_2 \left( 1 + \|\overline{X}^{(n)}_0\|+2M_2 \left\|B^H\right\|_{\lambda}|t|^\lambda( 1 + \|\overline{X}^{(n)}_0\|\right)  \right.\right.\notag\\
			& \hspace{1.5cm} \left. \left.+ C_1|t|^{ \lambda}\big(M_2C_\lambda M_1\left\|B^H\right\|_{\lambda}|t|^{\lambda} + M_2^2 \left\|B^H\right\|_{\lambda}|t|^\lambda + C_\lambda M_1\big) \right]\right]|t-s|^\lambda.
		\end{align}

		We will now complete the steps that allow us to obtain the inequalities in \eqref{eqinduccion}. Taking a small enough time $\psi$, in the interval $[0,t_{l+1}]\cap [0, \psi]$, by \eqref{eqE1-1} and  \eqref{DXst-2}, we have
		\begin{align}
			\label{calculoC1-0}
			\|  X^{(n)}\|_\lambda   &\leq   C_2  +  (C_4 + C_1C_3) \left\|B^H\right\|_{\lambda} |\psi|^{\lambda} 		\end{align}
		and
		\begin{align}
			\label{calculoC1-00}
			\| \delta \overline{X}^{(n)}(s, t)\| &\leq   [C_2  +  (C_4 + C_1C_3) \left\|B^H\right\|_{\lambda} |\psi|^{\lambda}] |t-s|^{ \lambda}
		\end{align}
		where $C_{2}= 3\|\xi\|_\lambda +  3M_2\|B^H\|_\lambda \left( 1 + \|\overline{X}^n_0\| \right)$, $C_{3}= 3C_\lambda M_1(M_2\left\|B^H\right\|_{\lambda}T^\lambda+1)+ 3M_2^2 \left\|B^H\right\|_{\lambda} T^\lambda$ and $C_4=  6M_2\|B^H\|_\lambda \left( 1 + \|\overline{X}^n_0\| \right)$. 
		
		To obtain some stability in the last norms estimates, we assume that  $  C_3 \left\|B^H\right\|_{\lambda}\psi^{\lambda} <  \frac{1}{2}$, if not we assume that 
		\begin{equation}
			\label{psi}
			\psi =\dfrac{1}{( 2 C_3(\left\|B^H\right\|_{\lambda}+1))^{1/\lambda}} 
		\end{equation}
		and also we take  the constant $C_1$ such that
		\begin{equation}
			\label{calculoC1-2}
			C_1 \ge 2\left(C_2 + C_4\left\|B^H\right\|_{\lambda} |\psi|^{\lambda}\right).
		\end{equation}
		Hence, on the interval $[0,t_{l+1}]\cap [0, \psi]$,  by \eqref{calculoC1-0}, \eqref{calculoC1-00}, \eqref{psi} and \eqref{calculoC1-2} we have
		\begin{equation*}
			\label{calculoC1-3}
			\|  X^{(n)}\|_\lambda  \leq C_1
		\end{equation*}
		and 
		\begin{align*}
			\label{calculoC1-31}
			\| \delta \overline{X}^{(n)}(s, t)\| &\leq   C_1 |t-s|^{ \lambda}.
		\end{align*}
		Thus, we have obtained the inequalities  \eqref{eqinduccion}.

		\ 
		
		\textbf{Step 4: Global upper bound for $\mathbf{\delta X^{n}}$.}
		Now, we extend the result from the interval $[0, \psi]$ to intervals of the form $[k\psi,(k+1)\psi]$. The calculus is similar, except for the fact that the initial values $X^{n}({0})$ and $\overline{X}^{(n)}_0$ must be updated to $X({k\psi})$ and $\overline{X}_{k\psi}$, respectively.  Note that the time step $\psi$ given by \eqref{psi} does not depend on $X({0})$ or $\overline{X}_0$. Hence it can be considered as a given constant. The interest constants are $C_{1}$ in inequality \eqref{calculoC1-2} and $K({X}^{(n)}_0, \overline{X}^{(n)}_0, C_1, f, B^H, t)$ in \eqref{cotaXnt}. \\
		
		We denote 
		%$t_0^{(k)}=k\psi$ and 
		$t_j^{(k)} $ as the $j$-term in $\left[\![ 0, T\right]\!]_n$ that belongs to $[k\psi, (k+1)\psi]$. 
		We consider
		\begin{equation*}
			{X}^{(n)}({k\psi})={X}^{(n)}({t^{(k)}_e}) \quad \text{and} \quad
			\overline{X}^{(n)}_{k\psi}=
			\overline{X}^{(n)}_{t^{(k)}_e} ,
		\end{equation*}
		where 
		\begin{equation*}
			t^{(k)}_e= \max \left[\![ 0,T\right]\!]\cap[0, k\psi], \quad k= 1, \cdots, n-1.
		\end{equation*}

		A straightforward generalization of \eqref{DXst-0}, for $t\in [\![ k\psi,(k+1)\psi ]\!] $, gives 
		\begin{align*}
			\label{DXst-0psi}
			|\delta X^{(n)}{(k\psi, t)}|  &\leq   \left(M_2(1 + \|\overline{X}^{(n)}_{k\psi}\|) + C_\lambda M_1C_1|t-k\psi|^{ \lambda}\right)\left\|B^H\right\|_{\lambda}|t-k\psi|^\lambda.
		\end{align*}
		Moreover, from 
		\eqref{DXs-01} and \eqref{DXs-02},
		\begin{equation*}\label{a1}
			|X^{(n)}({(k+1)\psi})|
			\le
			A_1 |\overline{X}^{(n)}_{k\psi}| + A_2
		\end{equation*}
		and 
		\begin{equation*}\label{a11}
			\|\overline{X}^{(n)}_{(k+1)\psi}\|
			\le
			A_1 \|\overline{X}^{(n)}_{k\psi}\| + A_2,
		\end{equation*}
		where $A_1=1+ 2M_2\left\|B^H\right\|_{\lambda}\psi^\lambda$ and  $A_2=\left(2M_2 +  C_\lambda M_1C_1|\psi|^{ \lambda}+M_2C_1|\psi|^{ \lambda} \right)\left\|B^H\right\|_{\lambda}\psi^\lambda$.
		Therefore, by induction procedure, we get 
		\begin{equation}\label{a2}
			\|\overline{X}^{(n)}_{k\psi}\|
			\le
			A_1^{k} \|\overline{X}_{0}^{(n)}\| + \frac{A_1^{k}-1}{A_1-1} A_2.
		\end{equation}
		
		Defining $T_{{\rm max}}$ as the number of intervals $[k\psi,(k+1)\psi]$ necessary to cover    $[0,T]$, from \eqref{psi}
		\begin{equation*}
			T_{{\rm max}}=\frac{T}{\psi}  =T \left({ 2}C_3 (\left\|B^H\right\|_{\lambda}+1)\right)^{\frac{1}{\lambda}}.
		\end{equation*}
		
		For all $k\le T_{{\rm max}}$,  \eqref{a2} implies
		\begin{equation}
			\label{cotaXkpsi}
			\|\overline{X}_{k\psi}^{(n)}\|
			\leq
			c_{1}(1+\|\overline{X}_{0}^{(n)}\|) \exp\left(  c_{2}(1+\|B^H\|_{\lambda}^{1/\lambda}) \right).
		\end{equation}
		In consequence,
		\begin{equation*}
			|{X}^{(n)}({k\psi})|
			\leq
			c_{1}(1+\|\overline{X}_{0}^{(n)}\|) \exp\left(  c_{2}(1+\|B^H\|_{\lambda}^{1/\lambda}) \right), \quad \text{for all} \ 0\le k\leq T_{{\rm max}}.
		\end{equation*}
		
		Finally, taking $s, t \in[0,T]$, we assume that $k\psi \leq s \leq (k+1)\psi \leq l\psi \leq t \leq (l+1)\psi$ with $l+1\leq  T_{{\rm max}}$, 
		\begin{eqnarray*}
			|\delta({X}^{(n)}({s, t}))|
			&\leq&
			|\delta({X}^{(n)}({s, (k+1)\psi})| +|\delta({X}^{(n)}({(k+1)\psi, (k+2)\psi})|+\cdots+ |\delta({X}^{n}({l\psi}, t)|\\
			&\leq& C_1(l-k+1)|t-s|^{\lambda}\\
			&\leq &C_1T_{{\rm max}}|t-s|^{\lambda}
		\end{eqnarray*}
		
		and similarly 
		\begin{eqnarray*}
			\|\delta(\overline{X}^{(n)}({s, t}))\|
			&\leq &C_1T_{{\rm max}}|t-s|^{\lambda}
		\end{eqnarray*}
		Hence, from  \eqref{calculoC1-2} and \eqref{cotaXkpsi}, we show that \eqref{eq:bnd-holder-Yn} and \eqref{eq:bnd-holder-Yn2} are true.
		\qed
	\end{proof}

	\section{Convergence of the Euler scheme}
	\label{sec:conv}
	Now, we prove that the Euler scheme introduced in Section \ref{Euler} converges to the solution of \eqref{eq1} with respect to the uniform norm. Towards this end, we study the behavior of the error processes $Z^{(n)}:= {X}-{X}^{(n)}$ and $\overline{Z}^{(n)}:= \overline{X}-\overline{X}^{(n)}$; 
	as well as  its increments.
	
	From \eqref{HolderX} and \eqref{eeuler1}, we have
	\begin{align*}
		\label{deltaZi}
		\delta Z^{(n)}(t_i, t_{i+1})&= (\delta X - \delta {X}^{(n)})(t_i, t_{i+1})\notag\\
		&=[f(\overline{X}_{t_i})- f(\overline{X}_{t_i}^{(n)})]\delta B^H(t_i, t_{i+1}) + R^{X}(t_i t_{i+1})\notag\\
		&= \sigma(t_i)\delta B^H(t_i, t_{i+1}) + R^{X}(t_i ,t_{i+1}),
	\end{align*}
	where $\sigma:[0,T]\to\mathbb{R}$ is the  function given by 
	$\sigma(u)=f(\overline{X}_u)- f(\overline{X}_{u}^{(n)}).$

	To analyze the convergence of the numerical scheme implemented in this Section \ref{Euler}, we impose an additional assumption.
	\begin{hypothesis} 	\label{H3} Let $\psi_i$, $i=1,\ldots,4,$ be four functions in $C^\lambda([-{\tau},0];
		\mathbb{R})$. Then, there exist $M_1,C>0$ such that   
		\begin{eqnarray*}
			|f(\psi_1)&-&f(\psi_2)-[f(\psi_3)-f(\psi_4)]|\\
			&\le&M_1\left\|\psi_1-\psi_2-[\psi_3-\psi_4]\right\|+C\left\|\psi_3-\psi_4\right\|
			\left[\|\psi_1-\psi_2-[\psi_3-\psi_4]\|+\left\|\psi_2-\psi_4\right\|\right].
		\end{eqnarray*}
	\end{hypothesis}
	\begin{remark}
		León and Tindel \cite{LeonTindel} have pointed out that Hypotheses \ref{H1} and \ref{H2} hold
		in the case that the coefficient $f$ in equation \eqref{eq1} has the form
		$$f(\psi)=\sigma\left(\int_{-{\tau}}^0\psi(\theta)\nu(d\theta)\right),\quad \psi\in C^\lambda([-{\tau},0];
		\mathbb{R}).$$
		Here $\sigma:\mathbb{R}\to\mathbb{R}$ is a continuous function with four bounded derivatives
		and $\nu$ is a finite signed measure on $[-\tau,0]$. Moreover, we claim that Hypothesis \ref{H3}
		also holds in this case. Indeed, the fundamental theorem of calculus implies 
		\begin{eqnarray*}
			|f(\psi_1)&-&f(\psi_2)-[f(\psi_3)-f(\psi_4)]|\\
			&=&\left| \int_0^1\sigma'\left(r\int_{-{\tau}}^0\psi_1(\theta)\nu(d\theta)+(1-r)\int_{-\tau}^0\psi_2(\theta)\nu(d\theta)\right)\int_{-\tau}^0\left(\psi_1(\theta)-\psi_2(\theta)\right)\nu(d\theta)dr\right.\\
			&&-\left.\int_0^1\sigma'\left(r\int_{-\tau}^0\psi_3(\theta)\nu(d\theta)+(1-r)\int_{-\tau}^0\psi_4(\theta)\nu(d\theta)\right)\int_{-\tau}^0\left(\psi_3(\theta)-\psi_4(\theta)\right)\nu(d\theta)dr
			\right|\\
			&\le&\left| \int_0^1\sigma'\left(r\int_{-\tau}^0\psi_1(\theta)\nu(d\theta)+(1-r)\int_{-\tau}^0\psi_2(\theta)\nu(d\theta)\right)\right.\\
			&&\times\left.\int_{-\tau}^0\left(\psi_1(\theta)-\psi_2(\theta)-[\psi_3(\theta)-\psi_4(\theta)]\right)
			\nu(d\theta)dr\right|\\
			&&+\left| \int_0^1\left[\sigma'\left(r\int_{-\tau}^0\psi_1(\theta)\nu(d\theta)+(1-r)\int_{-\tau}^0\psi_2(\theta)\nu(d\theta)\right)\right.\right.\\
			&&\left.\left.-\sigma'\left(r\int_{-\tau}^0\psi_3(\theta)\nu(d\theta)+(1-r)\int_{-\tau}^0\psi_4(\theta)\nu(d\theta)\right)\right]
			\int_{-\tau}^0\left(\psi_3(\theta)-\psi_4(\theta)\right)\nu(d\theta)dr\right|.
		\end{eqnarray*}
		Hence, it is now easy to see that our claim is true using the mean value theorem.
		
		Note that $f$ has an extension $g:C([-\tau,0];\mathbb{R})\to\mathbb{R}$ 
		given by
		$$g(\varphi)=\sigma\left(\int_{-\tau}^0\varphi(\theta)\nu(d\theta)\right),\quad \varphi\in C([-\tau,0];
		\mathbb{R}).$$
		It is easy to see that the mean value theorem implies that $g$ is  Fr\'echet differentiable with
		$$Dg(\varphi_1)(\varphi_2)=\sigma'\left(\int_{-\tau}^0\varphi_1(\theta)\nu(d\theta)\right)\int_{-\tau}^0\varphi_2(\theta)\nu(d\theta),\quad \varphi_1,\varphi_2\in C([-\tau,0];\mathbb{R}),$$
		and that there exists a constant $M>0$ such that
		$$\|Dg(\varphi_1)-Dg(\varphi_2)\|_{\mathcal{L}(C([0,-\tau];\mathbb{R}))}\le  M\|\varphi_1-\varphi_2\|,
		\quad \varphi_1,\varphi_2\in C([-\tau,0];\mathbb{R}).$$
		
		In general, Hypothesis \ref{H3} is satisfied if $f$ has an extension $g\in C^1(C([-\tau,0];\mathbb{R});\mathbb{R})$ with a bounded Fréchet derivative satisfying last inequality.
	\end{remark}

	On the other hand,	for $s, t \in S_2([\![ 0, \tau ]\!])$, 
	\begin{eqnarray}
		\label{deltaZ}
		\delta Z^{(n)}(s, t)=\sigma(s)\delta B^H(s, t) + R^{(n)}(s,t),
	\end{eqnarray}
	where
	\begin{equation}
		\label{deltaRZ}
		R^{(n)}(s,t):= \delta Z^{(n)}(s, t) - \sigma(s)\delta B^H(s, t).
	\end{equation}
	Notice that $ R^{(n)}(t_i,t_{i+1})= R^X(t_i,t_{i+1})$ for all $0\leq i \leq n-1$. Applying the operator $\delta$ on both sides of \eqref{deltaRZ} (remember that $\delta\delta=0$),
	\begin{equation}
		\label{deltaRZ2}
		\delta R^{(n)}(s,u,t)= -\delta \sigma(s, u)\delta B^H(u, t)
	\end{equation}

	Let $0 \leq u<s<t \leq T$. By \eqref{HolderX}, we obtain 
	$$\int_u^t f(\overline{X}_v) dB(v) = \int_u^s f(\overline{X}_v) dB(v)+\int_s^t f(\overline{X}_v) dB(v) $$
	and
	$$	f(\overline{X}_u)\delta B(u, t)+R^X(u,t) = f(\overline{X}_u)\delta B(u, s)+R^X(u,s)+f(\overline{X}_s)\delta B(s, t)+R^X(s,t).$$
	Consequently,
	\begin{align}\label{relres}
		R^X(u,t)-R^X(u,s) &=-f(\overline{X}_u)\delta B(u, t)+ f(\overline{X}_u)\delta B(u, s)+f(\overline{X}_s)\delta B(s, t)+R^X(s,t) \notag \\
		&=R^X(s,t)+(f(\overline{X}_s)-f(\overline{X}_u))\delta B(s, t).
	\end{align}

	\begin{theorem}
		Assume that Hypotheses \ref{H1}, \ref{H2} and \ref{H3} are satisfied. Also,
		let $Z^{(n)}:= {X}-{X}^{(n)}$ and  $\overline{Z}^{(n)}:= \overline{X}-\overline{X}^{(n)}$ where $X$ is the solution to equation \eqref{eq1} and $X^{(n)}$ is the Euler scheme defined in \eqref{euler}. For $\lambda$ given in Hypothesis \ref{H1}, $0 <  \varepsilon <2\lambda-1$ and $s, t \in[0,T]$ with $s<t$, we have
		\begin{equation*}\label{eq:bnd-holder-Zn}
			|Z^{(n)}_t| \leq \frac{\hat{C}_1}{n^{2\lambda-1-\varepsilon}},  \quad   
			|\delta Z^{(n)}{(s, t)}| \leq \frac{\hat{C}_2}{n^{2\lambda-1-\varepsilon}}|t-s|^\lambda,
		\end{equation*}
		\begin{equation*}\label{eq:bnd-holder-Zn2}
			\| \overline{Z}^{(n)}_t
			\| \leq  \frac{\hat{C}_1}{n^{2\lambda-1-\varepsilon}}  \quad  \text{and} \quad \|\delta \overline Z^{(n)}({s, t})\| \leq \frac{\hat{C}_2}{n^{2\lambda-1-\varepsilon}}|t-s|^\lambda,
		\end{equation*}    
		where $\hat{C}_1$ and $\hat{C}_2$ are positive constants such that $\hat{C}_1\leq c_1\exp(c_2(1+ \|B^H\|^{1/\lambda}_\lambda))$ and $\hat{C}_2\leq c_3\exp(c_4(1+ \|B^H\|^{1/\lambda}_\lambda))$, for some positive constants  $c_1, \cdots, c_4$ depending on $f$ and $\varepsilon$. 
	\end{theorem}
	
	\begin{proof}
		We follow steps similar to those used in the proof of Lemma \ref{lema2}. 
		We will prove by induction on $l$ that, over the interval $[0,t_{l+1}]\cap [0, \varsigma]$
		\begin{equation}
			\label{inductionZ1}
			|Z^{(n)}(0)|\leq \frac{d_1}{n^{2\lambda - 1 - \varepsilon}}, \quad |Z^{(n)}(s)|\leq \frac{C_1}{n^{2\lambda - 1 - \varepsilon}}, \quad |\delta Z^{(n)}(s,u)|\leq \frac{C_2}{n^{2\lambda - 1 - \varepsilon}}|u-s|^\lambda;
		\end{equation}
		and
		\begin{equation}
			\label{inductionZ2}
			\|\overline{Z}^{(n)}_0\|\leq \frac{d_1}{n^{2\lambda - 1 - \varepsilon}}, \quad \|\overline{Z}^{(n)}_s\|\leq \frac{C_1}{n^{2\lambda - 1 - \varepsilon}}, \quad \|\delta \overline{Z}^{(n)}(s,u)\|\leq \frac{C_2}{n^{2\lambda - 1 - \varepsilon}}|u-s|^\lambda,
		\end{equation}
		where the constant $\varsigma $ will be specified later, $t_l$ is given two lines after \eqref{defdelta}, $0< \varepsilon < 2\lambda-1$, $d_1,$ $C_1$ and $C_2$ are positive constants with $d_1\leq C_1$.
		
		First, we obtain that
		\begin{equation}
			\label{cotaz0n}
			|Z^{(n)}(0)|= |{X}(0)-{X}^{(n)}(0)| = |\xi(0)-\xi(0)|\leq \frac{d_1}{n^{2\lambda - 1 - \varepsilon}}
		\end{equation}
		and, consequently, 
		$$%	\label{cotazb0n}
		\|\overline{Z}^{(n)}_0\| = \|\overline{X}_0-\overline{X}^{(n)}_0\|= \sup_{\theta \in  [-\tau, 0]}|\xi(\theta)-\xi(\theta)|
		\leq \frac{d_1}{n^{2\lambda - 1 - \varepsilon}}.$$

		\textbf{Step 1} $l=1$.  Let $0\leq t\leq t_1$.  Since ${X}(0)={X}^{(n)}(0)$ and from \eqref{HolderX},
		\begin{align*}
			Z^{(n)}(t)&= X(t)-X^{(n)}(t) \notag\\
			&=X(0)+ \int_0^t f(\overline{X}_s)dB(s)-X^{(n)}(0)-f(\overline{X}^{(n)}_0)\delta B(0, t)
			=R^{X}(0,t).
		\end{align*}
		Hence, if $s,t \in [0,t_1], s<t$, we have, by Hypothesis \ref{H1}, Lemma \ref{lem1},  \eqref{HolderX2} and \eqref{relres},
		\begin{align}
			\label{dZb0t1}
			|\delta Z^{(n)}(s, t)|&=  |R^{X}(0,t)-R^{X}(0,s)|\notag \\
			&\leq |(f(\overline{X}_s)-f(\overline{X}_0))\delta B(s, t)| + |R^{X}(s,t)|\notag \\
			&\leq M_1\hat{c} s^{\lambda}\|B^H\|_\lambda|t-s|^{\lambda}+ C_X|t-s|^{2\lambda}  \notag\\
			&\leq \frac{\tau^\lambda (M_1\hat{c}\|B^H\|_\lambda+C_X)}{n^{2\lambda - 1 - \varepsilon}}|t-s|^{\lambda} 
		\end{align}
		and by \eqref{cotaz0n}
		\begin{align}
			|Z^{(n)}(t)|&\leq |Z^{(n)}(0) | + |\delta Z^{(n)}(0, t)|   \notag  \\
			&\leq \frac{d_1}{n^{2\lambda - 1 - \varepsilon}} + \frac{1}{n^{2\lambda - 1 - \varepsilon}}[ \tau^\lambda (M_1\hat{c}\|B^H\|_\lambda+C_X)]t_1^{\lambda} \notag \\
			&\leq  \frac{1}{n^{2\lambda - 1 - \varepsilon}}[d_1 + \tau^{2\lambda} (M_1\hat{c}\|B^H\|_\lambda+C_X) ].\notag
		\end{align}

		Now, we will study $\| \delta \overline{Z}^{(n)}(s, t)\|, $
		
		\begin{align*}
			\| \delta \overline{Z}^{(n)}(s, t)\|&=  \| \overline{Z}^{(n)}_{t}- \overline{Z}^{(n)}_{s}\| 
			=  \sup_{\theta \in  [-\tau, 0]} | \overline{Z}^{(n)}_{t}(\theta)- \overline{Z}^{(n)}_{s}(\theta)| \nonumber\\
			& = \sup_{\theta \in  [-\tau, 0]} | ({X}({t} + \theta)-{X}^{(n)}({t}+\theta))-({X}({s} + \theta)-{X}^{(n)}({s}+\theta))| \nonumber\\ 
			& = \sup_{\theta \in  [-\tau, 0]}  A(\theta) .   
		\end{align*}
		
		\textbf{Case 1.} $t \leq -\theta$, then 
		$$A(\theta)= |(\xi({t} + \theta) - \xi({t} + \theta))-(\xi({s} + \theta) - \xi({s} + \theta)) |=0.$$
		
		\textbf{Case 2.}  $t \geq -\theta$ and $s\leq -\theta$. Then, by \eqref{dZb0t1}, we have
		\begin{align*}
			A(\theta)&= |({X}({t} + \theta)-{X}^{(n)}({t}+\theta))-(\xi({s} + \theta) - \xi({s} + \theta)) | \notag\\
			&=|\delta Z^{(n)}(0, t+\theta)|\notag\\
			&\leq \frac{1}{n^{2\lambda - 1 - \varepsilon}}[ \tau^\lambda (M_1\hat{c}\|B^H\|_\lambda+C_X)]|t+\theta|^{\lambda}\notag\\
			&\leq \frac{1}{n^{2\lambda - 1 - \varepsilon}}[ \tau^\lambda (M_1\hat{c}\|B^H\|_\lambda+C_X)]|t-s|^{\lambda}.
		\end{align*}
		
		\textbf{Case 3.}  $s\geq -\theta$.  Due to  \eqref{dZb0t1} again, 
		\begin{align*}
			A(\theta)& \leq \frac{1}{n^{2\lambda - 1 - \varepsilon}}[  \tau^\lambda (M_1\hat{c}\|B^H\|_\lambda+C_X)]|t-s|^{\lambda}.
		\end{align*}
		Therefore, 	
		$$\| \delta \overline{Z}^{(n)}(s, t)\| \leq \frac{1}{n^{2\lambda - 1 - \varepsilon}}[ \tau^\lambda (M_1\hat{c}\|B^H\|_\lambda+C_X)]|t-s|^{\lambda}$$
		and
		\begin{align}
			\| \overline{Z}^{(n)}(t)\|&\leq \|\overline{Z}^{(n)}(0) \| + \|\delta \overline{Z}^{(n)}(0, t)\|   \notag  \\
			&\leq \frac{ d_1}{n^{2\lambda - 1 - \varepsilon}} +\frac{1}{n^{2\lambda - 1 - \varepsilon}}[ \tau^\lambda (M_1\hat{c}\|B^H\|_\lambda+C_X)]t^{\lambda}\notag \\
			&\leq \frac{ 1}{n^{2\lambda - 1 - \varepsilon}}[d_1  + \tau^{2\lambda} (M_1\hat{c}\|B^H\|_\lambda+C_X)]. \notag  
		\end{align}
		
		\textbf{Step 2: Upper bound for $\mathbf{R^n}$}
		
		We assume that the induction hypothesis (\ref{inductionZ1}) and \eqref{inductionZ2} hold for all $s, t \in [0,t_l] \cap [0, \varsigma]$. Here, we bound $R^{(n)}$ and $\delta R^{(n)}$.

		For $s<u<t=t_{l+1}\le\varsigma$, with $s,u,t \in \left[\![ 0, t_{l+1} \right]\!]_n$,  it follows from equation \eqref{deltaRZ2} and  Hypotheses \ref{H1} and \ref{H3} that
		\begin{align*}%	\label{eqRdeltaz}
			|\delta R^{(n)}(s,u,t)| &= |\delta \sigma(su)\delta B^H(u,t) |\notag \\
			& \leq  \|B^H\|_\lambda|t-u|^\lambda\left[M_1\|\delta\overline{Z}^{(n)}(s,u)\|+C
			\|\overline{Z}_s^{(n)}\|\left(\|\delta\overline{Z}^{(n)}(s,u)\|+\|\delta\overline{X}^{(n)}(s,u)\|
			\right)\right].
		\end{align*}
		Therefore, the induction assumption and Lemma \ref{lema2} yield
		\begin{align}
			\label{eqdRn}
			|\delta R^{(n)}(s,u,t)|&\leq \|B^H\|_\lambda|t-u|^\lambda\left[\frac{M_1C_2}{n^{2\lambda - 1 - \varepsilon}}|u-s|^{\lambda} + C \frac{C_1}{n^{2\lambda-1-\varepsilon}}|u-s|^\lambda \left[\frac{C_2}{n^{2\lambda-1-\varepsilon}}+ \hat{c} \right]\right] \notag \\
			&\leq  \frac{\|B^H\|_{\lambda}}{n^{2\lambda-1-\varepsilon}} |t-s|^{2\lambda} \left[M_1C_2 + CC_1 \left(C_2+ \hat{c} \right)\right] \notag \\
			&\leq  \frac{\|B^H\|_{\lambda}}{n^{2\lambda-1-\varepsilon}}  \left[M_1C_2 + CC_1 (C_2+ \hat{c} )\right] \varsigma^{2\lambda-1-\varepsilon} |t-s|^{1+\varepsilon} \notag \\
			&=  \frac{C_3}{n^{2\lambda-1-\varepsilon}} |t-s|^{1+\varepsilon}, 
		\end{align}
		where
		\begin{equation}
			\label{C3}
			C_3= \|B^H\|_{\lambda} \left[M_1C_2 + CC_1 (C_2+ \hat{c} )\right]\varsigma^{2\lambda-1-\varepsilon}.
		\end{equation}
		Using \eqref{HolderX2} and \eqref{deltaRZ}, we get
		
		\begin{equation*}
			|R^{(n)}(t_i,t_{i+1})|= |R^X(t_i,t_{i+1})|\leq C_X (t_{i+1}-t_i)^{2\lambda}.
		\end{equation*}
		Thus,
		$$ \frac{|R^{(n)}(t_i,t_{i+1})|}{|t_{i+1}- t_i|^{1+\varepsilon}} \leq \frac{C_X\tau^{2\lambda-1-\varepsilon} }{n^{2\lambda-1-\varepsilon}}.$$
		
		From Lemma \ref{lemswing} (Statement$(ii)$) and \eqref{eqdRn},
		\begin{equation}
			\label{C4}
			\|R^n\|_{[\![0, \varsigma]\!],1+\varepsilon}\leq \frac{C_{1+\varepsilon}C_3 + C_X\tau^{2\lambda-1-\varepsilon}}{n^{2\lambda-1-\varepsilon}}=\frac{C_4}{n^{2\lambda-1-\varepsilon}}
		\end{equation}
		thus we have obtained  a global bound for $\|R^n\|_{[0, \varsigma],1+\varepsilon}$. 
		
		\
		
		\textbf{Step 3: Upper bounds for $\mathbf{\delta {Z}^{(n)}}$ and $\mathbf{\delta \overline{Z}^{(n)}}$}
		
		Hypothesis \ref{H1},  \eqref{deltaZ},   \eqref{inductionZ2} and \eqref{C4}, for $0\leq s  < t=t_{l+1} \leq \varsigma, s,t \in \left[\![ 0, t_{l+1} \right]\!]_n$ lead us to
		\begin{align}
			\label{deltazn}
			| \delta Z^{(n)}(s,t)|&\leq |\sigma(s)||\delta B^H(s,t)| + |R^n(s,t)|\notag\\
			&\leq M_1 \|\overline{Z}^{(n)}_s\| \, \| B^H\|_\lambda|t-s|^{\lambda} + \frac{C_4}{n^{2\lambda-1-\varepsilon}}|t-s|^{1+\varepsilon}\notag\\
			&\leq   M_1\frac{C_1}{n^{2\lambda - 1 - \varepsilon}}\, \| B^H\|_\lambda|t-s|^{\lambda} + \frac{C_4}{n^{2\lambda-1-\varepsilon}}|t-s|^{1+\varepsilon}\notag\\
			&= \frac{M_1C_1\, \| B^H\|_\lambda \,  + C_4\varsigma^{1+\varepsilon-\lambda}}{n^{2\lambda - 1 - \varepsilon}}|t-s|^{\lambda} 
			= \frac{C'_5}{n^{2\lambda - 1 - \varepsilon}}|t-s|^{\lambda},
		\end{align}
		where 
		$$	%\label{C5}
		C'_5:= M_1C_1\, \| B^H\|_\lambda \,  + C_4\varsigma^{1+\varepsilon-\lambda}.
		$$

		Now, let $t_i \leq s<t \leq t_{i+1}, i \in \{0,1, \hdots, l\}$. Then, by \eqref{relres}
		\begin{align}\label{eqZST}
			& \delta Z^{(n)}(s,t)\notag\\
			&= \left(X(t_i)+\int_{t_i}^t f(\overline{X}_u)dB_u - X^{(n)}(t_i)-f(\overline{X}^{(n)}_{t_i})(B_t-B_{t_i})\right) \notag\\
			&\quad -\left(X(t_i)+\int_{t_i}^s f(\overline{X}_u)dB_u -X^{(n)}(t_i)-f(\overline{X}^{(n)}_{t_i})(B_s-B_{t_i})\right)\notag \\
			&= -f(\overline{X}^{(n)}_{t_i})(B_t-B_{s})+f(\overline{X}_{t_i})(B_t-B_{t_i})+R^X(t_i, t) \notag\\
			&\quad -f(\overline{X}_{t_i})(B_s-B_{t_i})-R^X(t_i , s)\notag \\
			&=(f(\overline{X}_{t_i})-f(\overline{X}^{(n)}_{t_i}))(B_t-B_{s})+(R^X(t_i , t)-R^X(t_i , s)) \notag \\
			&=(f(\overline{X}_{t_i})-f(\overline{X}^{(n)}_{t_i}))(B_t-B_{s})+R^X(s , t)+(f(\overline{X}_s)-f(\overline{X}_{t_i}))(B_t-B_{s}).
		\end{align}
		From the induction hypothesis, Hypothesis \ref{H1}, Theorem \ref{teo solution}   and \eqref{HolderX2}, we conclude
		\begin{align}\label{estZSTgen}
			|\delta Z^{(n)}(s,t)| &\le
			C_X|t-s|^{2\lambda}+M_1\hat{c}\| B^H\|_\lambda|s-t_i|^{\lambda}|t-s|^{\lambda}+M_1\|\overline{Z}^{(n)}_{t_i}\|\| B^H\|_\lambda|t-s|^{\lambda} \notag\\
			&\leq\frac{1}{n^{2\lambda - 1 - \varepsilon}}[ (C_X+M_1\hat{c}\|B^H\|_\lambda)\varsigma^{\lambda}]|t-s|^{\lambda} +\frac{ M_1C_1}{n^{2\lambda - 1 - \varepsilon}}\| B^H\|_\lambda|t-s|^{\lambda}.
		\end{align}

		Finally, in this step, we deal with the case  $0\leq s \leq t_i < t_j \leq t \leq t_{l+1}$, where $t_i,t_j \in \left[\![ 0, t_{l+1} \right]\!]_n$, are such that  $t_i$ is the smallest member of the partition bigger than $s$, while $t_j$ is the biggest member of the partition smaller than $t$. Then, \eqref{deltazn}  and  
		\eqref{estZSTgen} give
		\begin{align}\label{deltaZstgen}
			|\delta Z^{(n)}(s,t)| &\leq |\delta Z^{(n)}(s,t_i)|+|\delta Z^{(n)}(t_i, t_j)|+|\delta Z^{(n)}(t_j,t)|  \notag\\
			&\leq \frac{1}{n^{2\lambda - 1 - \varepsilon}}[ (C_X+M_1\hat{c}\|B^H\|_\lambda)\varsigma^{\lambda}+M_1C_1\| B^H\|_\lambda]|t_i-s|^{\lambda} \notag \\
			&\quad + \frac{C'_5}{n^{2\lambda - 1 - \varepsilon}}|t_j-t_i|^{\lambda} \notag \\
			&\quad + \frac{1}{n^{2\lambda - 1 - \varepsilon}}[ (C_X+M_1\hat{c}\|B^H\|_\lambda)\varsigma^{\lambda}+M_1C_1\| B^H\|_\lambda]|t-t_j|^{\lambda} \notag \\
			&\leq \frac{C_5}{n^{2\lambda - 1 - \varepsilon}}|t-s|^{\lambda},
		\end{align}
		where  
		\begin{equation}
			\label{c5}
			C_5:= 2(C_X+M_1\hat{c}\|B^H\|_\lambda)\tau^{\lambda} + 3M_1C_1\, \| B^H\|_\lambda \,  + C_4\varsigma^{1+\varepsilon-\lambda}.
		\end{equation}

		It is the turn of 		$\overline{Z}^{(n)}$:
		\begin{align*}
			\| \delta \overline{Z}^{(n)}(s,t)\|&=  \| \overline{Z}^{(n)}_{t}- \overline{Z}^{(n)}_{s}\| 
			=  \sup_{\theta \in  [-\tau, 0]} | \overline{Z}^{(n)}_{t}(\theta)- \overline{Z}^{(n)}_{s}(\theta)| \nonumber\\
			& = \sup_{\theta \in  [-\tau, 0]} | {X}({t} + \theta)-{X}^{(n)}({t}+\theta) - ({X}({s} + \theta)-{X}^{(n)}({s}+\theta))| \nonumber\\ 
			& = \sup_{\theta \in  [-\tau, 0]}  A(\theta).    
		\end{align*}
		The analysis of $A(\theta)$ is divided in three cases:
		
		\textbf{Case 1.}  Here, $-\tau \leq \theta \leq -t$. Thus, $t + \theta$ and $s+\theta$ are non-positive numbers, therefore
		$$A(\theta)= | \xi({t} + \theta)-\xi({t}+\theta) - (\xi({s} + \theta)-\xi({s}+\theta))|=0. $$
		
		\textbf{Case 2.} Now,  $-s \leq \theta \leq 0$. In consequence, $t + \theta$ and $s+\theta$ are non-negative numbers. Therefore,  \eqref{deltazn}, \eqref{eqZST} and \eqref{deltaZstgen} imply
		$$A(\theta)=|\delta Z^{(n)}(s+\theta, t+\theta)|\leq \frac{C_5}{n^{2\lambda - 1 - \varepsilon}}|t-s|^{\lambda}.$$
		
		\textbf{Case 3.} The last case is that  $-t  < \theta < -s$. Consequently,  $t + \theta>0$ and $s + \theta <0$. Then, using   \eqref{deltazn}, \eqref{eqZST} and \eqref{deltaZstgen}  again, we obtain
		\begin{align*}
			A(\theta) &= | {X}({t} + \theta)-{X}^{(n)}({t}+\theta) - (\xi({s} + \theta)-\xi({s}+\theta))|\\
			&= |\delta Z^{(n)}(0, t+\theta)|
			\leq \frac{C_5}{n^{2\lambda - 1 - \varepsilon}}|t+\theta|^{\lambda}
			\leq \frac{C_5}{n^{2\lambda - 1 - \varepsilon}}|t-s|^{\lambda}.
		\end{align*}
		Therefore, the last three cases allow us to conclude
		\begin{equation}
			\label{deltaznb}
			\| \delta \overline{Z}^{(n)}(s,t)\|\leq \frac{C_5}{n^{2\lambda - 1 - \varepsilon}}|t-s|^{\lambda}, \quad s,t\in[0,\varsigma].
		\end{equation}
		
		Finally, in this step, we figure out  an appropriate value for $\varsigma$ in order to obtain that \eqref{inductionZ1} and \eqref{inductionZ2} true. 
		
		From \eqref{deltaZstgen} and \eqref{deltaznb}, we can take $\varsigma$ such that $C_5\leq C_2$. Indeed, 
		by \eqref{C3}, \eqref{C4} and \eqref{c5}, we get 	
		\begin{align*}
			%\label{C5}
			C_5&= 2(C_X+M_1\hat{c}\|B^H\|_\lambda)\tau^{\lambda} + 3M_1C_1\, \| B^H\|_\lambda \,  + C_4\varsigma^{1+\varepsilon-\lambda}.\notag\\ 
			&= 2(C_X+M_1\hat{c}\|B^H\|_\lambda)\tau^{\lambda}+3M_1C_1\, \| B^H\|_\lambda + \left(C_{1+\varepsilon}C_3 + C_X\tau^{2\lambda-1-\varepsilon}\right)\varsigma^{1+\varepsilon-\lambda}\notag\\ 
			&= 2(C_X+M_1\hat{c}\|B^H\|_\lambda)\tau^{\lambda}+ 3M_1C_1\, \| B^H\|_\lambda \notag \\
			& \quad +  \left(C_{1+\varepsilon}(\|B^H\|_{\lambda} \left[M_1C_2 + CC_1 (C_2+ \hat{c} )\right]\varsigma^{2\lambda-1-\varepsilon}) + C_X\tau^{2\lambda-1-\varepsilon}\right)\varsigma^{1+\varepsilon-\lambda}\notag\\
			&=  2(C_X+M_1\hat{c}\|B^H\|_\lambda)\tau^{\lambda}+ 3M_1C_1\, \| B^H\|_\lambda \notag \\
			& \quad + C_{1+\varepsilon}\|B^H\|_{\lambda}CC_1\hat{c}\varsigma^{\lambda}   + C_X\tau^{2\lambda-1-\varepsilon}\varsigma^{1+\varepsilon-\lambda}  + C_{1+\varepsilon}\|B^H\|_{\lambda}C_2\left( M_1 + CC_1\right)\varsigma^{\lambda}. 
		\end{align*}        
		Hence, we select $\varsigma$ that satisfies 
		\begin{equation}
			\label{varsigma1}
			\varsigma^{\lambda}\leq \frac{1}{\|B^H\|_{\lambda}2 C_{1+\varepsilon}( M_1 + CC_1)}.
		\end{equation}
		Then, it is sufficient to have the inequality
		$$	\frac{C_2}{2}\geq  2(C_X+M_1\hat{c}\|B^H\|_\lambda)\tau^{\lambda}+ 3M_1C_1\, \| B^H\|_\lambda + C_{1+\varepsilon}\|B^H\|_{\lambda}CC_1\hat{c}\varsigma^{\lambda}   + C_X\tau^{2\lambda-1-\varepsilon}\varsigma^{1+\varepsilon-\lambda}.$$
		Assuming that $\varsigma\leq 1$ and choosing
		$$	C_2= 4(C_X+M_1\hat{c}\|B^H\|_\lambda)\tau^{\lambda}+ 6M_1C_1\, \| B^H\|_\lambda + 2C_{1+\varepsilon}\|B^H\|_{\lambda}CC_1\hat{c}  + 2C_X\tau^{2\lambda-1-\varepsilon},$$
		we have that $C_5\leq C_2$, obtaining the results relating to $\delta Z^{(n)}(s,t)$ and $\delta \overline{Z}^{(n)}(s,t)$ in \eqref{inductionZ1} and \eqref{inductionZ2}, respectively. 
	\end{proof}
	
	\
	
	\textbf{Step 4: Upper bound for $\mathbf{{Z}^{(n)}}$ and $\mathbf{\overline{Z}^{(n)}}$}
	
	In order to facilitate the notation of $C_2$, it is possible to define $C_2$ as 
	\begin{equation*}
		C_2=\alpha_1(1+C_1)(1+\| B^H\|_\lambda),
	\end{equation*}
	for some positive constant $\alpha_1$.

	Applying \eqref{inductionZ1} and \eqref{inductionZ2}  for $t\in [0,\varsigma]$, we obtain
	\begin{equation*}
		\|{Z}^{(n)}(t)\|\leq \|{Z}^{(n)}(0)\| + \|\delta{Z}^{(n)}(0,t)\|\leq \frac{\alpha_2}{n^{2\lambda - 1 - \varepsilon}}
	\end{equation*}
	and
	\begin{equation*}
		\|\overline{Z}^{(n)}_t\|\leq \|\overline{Z}^{(n)}_0\| + \|\delta\overline{Z}^{(n)}(0,t)\|\leq \frac{\alpha_2}{n^{2\lambda - 1 - \varepsilon}},
	\end{equation*}
	where $\alpha_2=d_1+\alpha_1(1+C_1)(1+\| B^H\|_\lambda)\varsigma^\lambda$.
	
	Taking $C_1>4d_1$ with $d_1>\frac14$ and
	\begin{equation}
		\label{varsigma2}
		\varsigma^\lambda \leq \frac{1}{4\alpha_1(1+\| B^H\|_\lambda)},
	\end{equation}
	we show that $a_2\leq C_1.$
	
	\
	
	\textbf{Step 5: Global bounds}
	
	Note that there exists a constant $\alpha_4>0$ such that
	\begin{equation}
		\label{varsigma3}
		\varsigma= \frac{\alpha_4}{(1+\| B^H\|_\lambda)^{1/\lambda}} ,
	\end{equation}
	and	the inequalities \eqref{varsigma1} and \eqref{varsigma2} hold.  Hence, taking 
	\begin{equation}
		\label{globalb5}
		C_1=4d_1, \quad d_1>\frac{1}{4} \quad \text{and} \quad C_2=\alpha_1(1+C_1)(1+\| B^H\|_\lambda),
	\end{equation}
	we obtain that the relations \eqref{inductionZ1} and \eqref{inductionZ2} are valid on $  [0, \varsigma]$.
	
	The next stage is to study our results in the intervals of the form $I_j=[j\varsigma, (j+1)\varsigma]$. We can assume that $\varsigma\in \left[\![ 0, T\right]\!]$ by selecting an appropriate $\alpha_4$ in \eqref{varsigma3}. The calculations for each $I_j$ are similar to those given in the previous steps, by updating the initial conditions ${Z}^{(n)}(0)$ and $\overline{Z}^{(n)}_0$  to ${Z}^{(n)}(j\varsigma)$ and $\overline{Z}^{(n)}_{j\varsigma}$, respectively. Therefore, the induction hypothesis on $I_j$ can be rewritten as
	\begin{equation*}
		%	\label{inductionZ1j}
		|Z^{(n)}(j\varsigma)|\leq \frac{d_{1,j}}{n^{2\lambda - 1 - \varepsilon}}, \quad |Z^{(n)}(s)|\leq \frac{C_{1,j}}{n^{2\lambda - 1 - \varepsilon}}, \quad |\delta Z^{(n)}(s,u)|\leq \frac{C_{2,j}}{n^{2\lambda - 1 - \varepsilon}}|u-s|^\lambda;
	\end{equation*}
	and
	\begin{equation*}
		%	\label{inductionZ2j}
		\|\overline{Z}^{(n)}_{j\varsigma}\|\leq \frac{d_{1,j}}{n^{2\lambda - 1 - \varepsilon}}, \quad \|\overline{Z}^{(n)}_s\|\leq \frac{C_{1,j}}{n^{2\lambda - 1 - \varepsilon}}, \quad \|\delta \overline{Z}^{(n)}(s,u)\|\leq \frac{C_{2,j}}{n^{2\lambda - 1 - \varepsilon}}|u-s|^\lambda.
	\end{equation*}
	Moreover, the relation \eqref{varsigma3} and \eqref{globalb5} can be replaced by
	\begin{equation*}
		%	\label{varsigma3j}
		\varsigma=	\varsigma({j+1})-j\varsigma= \frac{\alpha_4}{(1+\| B^H\|_\lambda)^{1/\lambda}} 
	\end{equation*}
	and
	\begin{equation}
		\label{globalb1}
		C_{1, j+1}=4C_{1,j} \quad \text{and} \quad C_{2,j}=\alpha_1(1+C_{1,j})(1+\| B^H\|_\lambda)
	\end{equation}
	respectively. 
	Now, iterating \eqref{globalb1}, we have
	\begin{equation}\label{d4}
		C_{1,j} = d_1 \, 4^{j}.
	\end{equation}
	
	From \eqref{varsigma3}, the number of iterations needed in order to fill $[0,T]$ is bounded by $\frac{(1+\|B^H\|_{\lambda})^{1/\lambda} \, T}{\alpha_{4}}$. 
	Then, from \eqref{globalb1} and \eqref{d4}, we get
	\begin{equation*}
		C_{1,j} \leq c_{1}\exp\left(  c_{2}(1+\|B^H\|_{\lambda}^{1/\lambda}) \right) \quad \text{and} \quad C_{2,j} \leq c_{3}\exp\left(  c_{4}(1+\|B^H\|_{\lambda}^{1/\lambda}) \right) 
	\end{equation*}
	for all $\ 0\le j < \frac{(1+\|B^H\|_{\lambda})^{1/\lambda} \, T}{\alpha_{4}}$. In this way, we have proven that the result is true.\qed
	%%%%%%%%%%%%%%%%%%%%%%%%%%%%%%%%%%%%%%%%%%%%%%%%%%%%%%%%%%%%%%%%%%%%%
	\begin{remark}
		Let $p \geq 1$ and $0 < \varepsilon <2\lambda-1$. Given that quadratic exponential moments for $\|B^H\|_{\lambda}$ exist, and since $1/\lambda<2$, by Theorem 2 we conclude
		\begin{equation*}
			\left(E\left[\|Z^{(n)}\|^p_{\lambda,[0,T]}\right] \right)^{1/p}\leq \frac{c_3 \left(E\left[\exp(pc_4(1+ \|B^H\|^{1/\lambda}_\lambda))\right]\right)^{1/p}}{n^{2\lambda-1-\varepsilon}}
			=\frac{C}{n^{2\lambda-1-\varepsilon}}. 
		\end{equation*}
		
	\end{remark}
	
	\section{Simulations}
	\label{sec:examples}
	Let us consider a one-dimensional fBm $B^H$ with Hurst parameter $H \in (1/2,1)$. We also consider the following stochastic functional equation
	\begin{eqnarray*}
		dX(t)&=& f(\overline{X}_t)dB^H(t), \quad t\in [0, T] \\
		\overline{X}_0&=& \xi \notag
	\end{eqnarray*}
	where $T=1$, $H=0.75$, $\tau=0.1$, and the initial condition $\xi(x): [-\tau, 0]\to \mathbb{R}$, given by $\xi(x)=x^2+2$ is a $\lambda$-H\"older continuous function for $\lambda\in (0,1)$. The functional $f$ is defined by 
	$$f(\psi)=\int_{-\tau}^0\psi(\theta)d\theta + \sin\left(\int_{-\tau}^0\psi(\theta)d\theta\right),\quad \psi\in C^\lambda([-\tau,0];
	\mathbb{R}).$$
	Under these conditions there exists a unique solution and the Euler scheme converges. To simulate the unique solution $X$, we set $\Delta = 500^{-1}$. FBm was generated by the algorithm proposed by Abry and Sellan \cite{abry} using the MATLAB code (wfbm). In addition, $r=\tau \cdot N$ corresponds to the number of time steps to the history memory $\tau$.

	\begin{figure}[H]
		\centering	\includegraphics[width=0.8\textwidth]{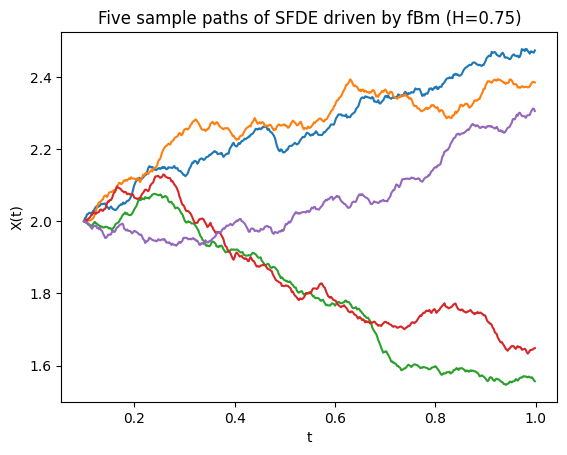}
		\label{lineal}
	\end{figure}

	\section*{Acknowledgements}
	Johanna Garz\'on was partially supported by HERMES project 58557 and Mathamsud EXPLORE-SDE AMSUD240037. J. Lozada was supported by the SECIHTI fellowship CVU 996983. Soledad Torres was partially supported by Fondecyt Reg. project number 1230807, 1221373, Mathamsud SMILE AMSUD230032, Mathamsud 
	EXPLORE-SDE AMSUD240037 and Mathamsud SiJaVol AMSUD240024. 
	This work was supported by Centro de Modelamiento Matemático (CMM) BASAL fund FB210005 for center of excellence from ANID-Chile.  \\

\end{document}